\documentclass[12pt,a4paper]{article}

\usepackage{amsmath,amsfonts,amsthm,amssymb}
\usepackage{mathrsfs} 
\usepackage{parskip}
\usepackage[T1]{fontenc}
\usepackage[utf8]{inputenc}
\usepackage{authblk}
\usepackage{enumerate}
\usepackage{graphicx}
\usepackage{xcolor}
\usepackage{a4wide}
\allowdisplaybreaks

\newtheorem{theorem}{Theorem}

\newtheorem{remark}{Remark}
\newtheorem{proposition}{Proposition}

\let\ds\displaystyle
\let\scr\mathscr

\let\goth\mathfrak

\def \Liminf{\mathop{\underline{\lim}}\limits}

\def\Pb{\mathbf{P}}

\def\Ex{\mathbf{E}}
\def\Pb{\mathbf{P}}

\def\UU{\mathbb{U}}

\def\1{\mbox{1\hspace{-.25em}I}}
\begin{document}
\title{Adaptive Kalman Filter for  Systems with Unknown Initial Values}
\author{
 \textsc{Yury A. Kutoyants}\\ {\small Le Mans University, Le Mans,  France }\\
{\small Tomsk     State University, Tomsk, Russia }\\
 }

\date{}

\maketitle
\begin{abstract}
   The models of
 partially observed linear stochastic differential equations with unknown
 initial values of the non-observed component are considered in two
 situations. In the first problem, the initial value is deterministic, and in
 the second problem, it is assumed to be a Gaussian random variable. The main
 problem is the computation of adaptive Kalman filters and the discussion of
 their asymptotic optimality. The realization of this program for both models
 is done in several steps. First, a preliminary estimator of the unknown
 parameter is constructed by observations on some learning interval. Then,
 this estimator is used for the calculation of recurrent one-step MLE
 estimators, which are subsequently substituted in the equations of Kalman
filtration.
\end{abstract}



\section{Introduction}

Consider the partially observed continuous time system
\begin{align}
\label{21-01}
{\rm d}X_t&=f\left(\vartheta ,t\right)Y_t{\rm d}t+\sigma \left(t\right){\rm
  d}W_t,\qquad X_0=0,\qquad 0\leq t\leq T,\\
{\rm d}Y_t&=a\left(\vartheta ,t\right)Y_t{\rm d}t+b\left(\vartheta
,t\right){\rm d}V_t,\qquad y_0,\quad 0\leq t\leq T, \label{21-02}
\end{align}     
where $W_t,0\leq t\leq T$ and $V_t,0\leq t\leq T$ are independent Wiener
processes, all functions $f\left(\cdot \right),\sigma \left(\cdot
\right),a\left(\cdot \right),b\left(\cdot \right) $ are known and $\vartheta $
is some unknown finite dimensional parameter. The process $X^T=\left(X_t,0\leq
t\leq T\right)$ is observed and the process $Y^T=\left(Y_t,0\leq t\leq
T\right)$ is ``hidden''. The main problem is to find a good approximation of
the unobserved component $Y_t$ by observations $X^t=\left(X_s,0\leq s\leq
t\right)$. If $\vartheta $ and $y_0$ are known then the solution of the
problem is well-known and was given by Kalman  and Bucy \cite{KB61}: the conditional
expectation $m\left(\vartheta ,t\right)=\Ex_\vartheta \left(Y_t|X^t\right)$
is solution of the system of equations
\begin{align}
\label{21-03}
{\rm d}m\left(\vartheta ,t\right)&=a\left(\vartheta ,t\right)m\left(\vartheta
,t\right) {\rm d}t+\frac{f\left(\vartheta ,t\right)\gamma \left(\vartheta
  ,t\right)}{\sigma \left(t\right)^2}\left[{\rm d}X_t-f\left(\vartheta
  ,t\right)m\left(\vartheta
,t\right) {\rm d}t\right] ,\\
\frac{{\rm d}\gamma \left(\vartheta ,t\right)}{{\rm d}t}&=2a\left(\vartheta
,t\right) \gamma \left(\vartheta ,t\right)-\frac{f\left(\vartheta
  ,t\right)^2\gamma \left(\vartheta ,t\right)^2}{\sigma
  \left(t\right)^2}+b\left(\vartheta ,t\right)^2 \label{21-04}
\end{align} 
subject to the initial conditions $m\left(\vartheta ,0\right)=y_0 $ and
$\gamma \left(\vartheta ,0\right)=0 $.  The construction of the adaptive
filters (with $\vartheta$ unknown and $y_0$ known) for both models is given in
two steps. First, the unknown parameter is estimated, and then this estimator
is substituted into the corresponding equations, such as \eqref{21-03} and
\eqref{21-04}. Due to its practical importance, there exists a vast literature
on adaptive filtering for discrete-time models; see \cite{CMR05},
\cite{Hay14}, \cite{S08}, and references therein. Continuous-time systems are
less well studied.

General results concerning parameter estimation and adaptive filtering can be
found, for example, in the works \cite{BaCr09}, \cite{CMR05}, \cite{Hay14},
\cite{LSv2}, and \cite{S08}.

The problems related with the construction of the filtering equations in the
situations where the initial values of the hidden system are unknown or wrong
occupy a special place in the filtering theory. The different statements and
results can be found in the works \cite{KSSQAS20}, \cite{ZSL18} (see as well
references therein).

The properties of estimators for continuous time models like
\eqref{21-01}, \eqref{21-02} are studied in different asymptotics. The simplest
case corresponds to {\it small noise} in both equations, i.e., where $\sigma
(t) \rightarrow \varepsilon \sigma (t)$,
$b(\vartheta, t) \rightarrow \varepsilon b(t)$, and
$\varepsilon \rightarrow 0$. The properties of parameter estimators were
described in many works, see, for example, \cite{Kon90}, \cite{Kut84},
\cite{Kut94}, \cite{Kut17}, \cite{KZ21}, \cite{SKh96}, and references
therein. The parameter estimation in the case of a homogeneous partially
observed linear system was studied in the asymptotic of {\it large samples} in
the works \cite{A83}, \cite{Ba81}, \cite{BaB84}, \cite{BK96},
\cite{DZ86}, \cite{KS91}, \cite{Kuri23}, \cite{Kut84},
\cite{Kut04}, \cite{Kut19b}, \cite{LSv2}. The third model of
observations: small noise in the observation equations only was studied in the
works \cite{GJ1-01}, \cite{GJ2-01}, \cite{KBS84}, \cite{Kut19a}, \cite{Kut24a}, \cite{R11}.
Remark that for these three models the question of asymptotically efficient
adaptive filtration was discussed in the works \cite{KZ21}, \cite{Kut24b}, and
\cite{Kut24a} respectively.

In this work we consider three problems of adaptive filtration in the cases
where the initial value of $Y_0=y_0$ is supposed to be unknown. Note that in the
mentioned above works \cite{Kut84}, \cite{Kut94}, \cite{KZ21} it was always
supposed that $y_0$ is known. The chosen models are relatively simple but nevertheless
the construction of adaptive filters is  non trivial. All filters
are constructed in four steps. The first step is to obtain a preliminary
consistent estimator of the unknown parameter, the second step - this
estimator is used for the construction of the recurrent One-Step MLE-process
 and the third step is the substitution of this One-step
estimator-process in the Kalman-Bucy equations. The solution of these
equations gives us the adaptive Kalman filter. The last step is to describe
the error of approximation of the conditional expectation. The similar
problems of asymptotically efficient estimation of $m\left(\vartheta
,t\right)$ were considered in \cite{KZ21} (small noise in both equations), in
\cite{Kut24a} (small noise in observations only), in \cite{Kut24b} (large
samples statement) and in \cite{Kut24c} (hidden autoregressive process).
The main advantage of the proposed procedure is its relative computational
simplicity because after the small learning interval the estimator of the
unknown parameter and the adaptive filter have recurrent structure. .

\section{Model of observations} 

 Consider the partially observed linear system of equations
\begin{align}
\label{21-1}
{\rm d}X_t&=f\left(t \right)Y_t\,{\rm d}t+\varepsilon
\sigma(t)\,{\rm d}W_t,\qquad X_0=0,\quad 0\leq t\leq T,\\
 {\rm  d}Y_t&=a\left(t \right)Y_t\,{\rm d}t+\varepsilon
b\left(t \right)\,{\rm d}V_t,\qquad\; Y_0 .
\label{21-2}  
\end{align}
where $W_t,V_t,0\leq t\leq T$ are independent Wiener processes. The functions
$f\left(t \right),a\left(t \right),$ $\sigma \left(t \right),$ $b\left(t
\right);t\in \left[0,T\right] $ and the parameter $\varepsilon\in (0,1]$ are
  supposed to be known.  The initial value $\vartheta \in\Theta $ is the
  unknown parameter.  The observations are $X^T=\left(X_t,0\leq t\leq
  T\right)$ and the process $Y^T=\left(Y_t,0\leq t\leq T\right)$ is
  ``hidden''. Here $\varepsilon \in(0,1]$ is some known small parameter and
    the asymptotic corresponds to the limit $\varepsilon \rightarrow
    0$. Therefore we consider a linear system with small noise in  both
    equations. 

\subsection{Deterministic  initial value}
\subsubsection{Unknown  initial value}
Suppose that the partially observed system is
\begin{align*}
{\rm d}X_t&=f\left(t\right)Y_t\,{\rm d}t+\varepsilon \sigma \left(t\right){\rm d}W_t,\qquad
\quad X_0=0,\qquad 0\leq t\leq T,\\
{\rm d}Y_t&=a\left(t\right)Y_t\,{\rm d}t+\varepsilon b \left(t\right){\rm d}V_t,\qquad\quad  Y_0=\vartheta ,
\end{align*}
where  the  functions $f\left(t\right),\sigma
\left(t\right),a\left(t\right),b\left(t\right) $ are known and  satisfy the

\bigskip

 Conditions ${\scr A}_*$.  {\it The functions $f\left(t \right),\sigma \left(t
\right),a\left(t \right),b\left(t \right),t\in \left[0,T\right]$  are bounded and the
functions $f\left(\cdot \right),\sigma \left(\cdot\right)$ are separated from 0.}

\bigskip

 The unknown parameter is the initial value $Y_0=\vartheta \in\Theta
=\left(\alpha,\beta \right)$ and we have to estimate $\vartheta $ by
observations $X^T$. Moreover this estimator will be used later for the
construction of the adaptive Kalman filter. Our goal is to obtain a recurrent
adaptive filter. We need as well  a recurrent presentation of the estimator,
because at each 
moment $t\in\left[0,T\right]$ we need an estimator constructed by the
observations $X_s,0\leq s\leq t$. Then this estimator will be used in the
construction of the filter. 

Recall that the conditional expectation $m\left(\vartheta
,t\right)=\Ex_\vartheta \left(Y_t |{\goth F}_t^X\right)$ (here ${\goth F}_t^X$
is the $\sigma $-algebra generated by the  observations
$X^t=\left(X_s,0\leq s\leq t\right)$) satisfies the equations of Kalman-Bucy
filtration \eqref{21-03},\eqref{21-04}  
\begin{align*}
{\rm d}m\left(\vartheta ,t\right)&=a\left(t\right)m\left(\vartheta
,t\right){\rm d}t+\frac{f\left(t\right)\gamma\left(t\right)}{\varepsilon ^2\sigma
  \left(t\right)^2}\left[{\rm d}X_t-f\left(t\right)m\left(\vartheta
  ,t\right){\rm d}t\right] ,\; m\left(\vartheta ,0\right)=\vartheta ,\\
\frac{\partial \gamma\left(t\right)}{\partial t}&=2a\left(t\right)\gamma
\left(t\right)-\frac{f\left(t\right)^2\gamma\left(t\right)^2}{\varepsilon ^2\sigma
  \left(t\right)^2}+\varepsilon ^2b\left(t\right)^2 ,\qquad \gamma\left(0\right)=0,\qquad
0\leq t\leq T.  
\end{align*}
These equations can be slightly simplified by introducing the function
$\gamma_* \left(t\right)=\varepsilon ^{-2}\gamma\left(t\right) $. Then we can
write
\begin{align}
\label{21-z1}
{\rm d}m\left(\vartheta ,t\right)&=a\left(t\right)m\left(\vartheta
,t\right){\rm d}t+\frac{f\left(t\right)\gamma_* \left(t\right)}{\sigma
  \left(t\right)^2}\left[{\rm d}X_t-f\left(t\right)m\left(\vartheta
  ,t\right){\rm d}t\right] ,\; m\left(\vartheta ,0\right)=\vartheta ,\\
\frac{\partial \gamma_* \left(t\right)}{\partial t}&=2a\left(t\right)\gamma_* 
\left(t\right)-\frac{f\left(t\right)^2\gamma_* \left(t\right)^2}{\sigma
  \left(t\right)^2}+b\left(t\right)^2 ,\qquad \gamma_* \left(0\right)=0,\qquad
0\leq t\leq T.  
\label{21-z2}
\end{align}
The solution of the equation \eqref{21-z1} is
\begin{align}
\label{21-z3}
m\left(\vartheta ,t\right)&=\vartheta \Phi \left(0,t\right)+\int_{0}^{t}\Phi
\left(s,t\right) {f\left(s\right)\gamma_* \left(s\right)}{\sigma
  \left(s\right)^{-2}}{\rm d}X_s=\vartheta \Phi \left(0,t\right)+H\left(t,X^t\right)
\end{align}
where
\begin{align*}
\Phi \left(s,t\right)&=\exp\left(\int_{s}^{t}A\left(r\right){\rm d}r\right),
\qquad A\left(t\right)=a\left(t\right)-  {f\left(t\right)^2\gamma_* \left(t\right)}{\sigma
  \left(t\right)^{-2}},\\
H\left(t,X^t\right)&=\int_{0}^{t}\Phi
\left(s,t\right) {f\left(s\right)\gamma_* \left(s\right)}{\sigma
  \left(s\right)^{-2}}{\rm d}X_s.
\end{align*}
Therefore the good  approximation of $m\left(\vartheta ,t\right)$ needs a good
estimator of the parameter $\vartheta $. 

One way to estimate $\vartheta $ is  use the MLE $\hat\vartheta
_{t,\varepsilon }$. The measures $\left(\Pb_\vartheta ^{\left(t,\varepsilon
  \right)},\vartheta \in\Theta\right) $ induced by observations $X^t$ in the
measurable space of continuous functions
$\left({\cal C}\left[0,t\right],{\goth B}_t\right)$  under condition ${\scr A}_*$ are
equivalent and the likelihood ratio (LR) function is \cite{LSv1}
\begin{align}
\label{21-LR}
L\left(\vartheta ,X^t\right)&=\exp\left(\int_{0}^{t}\frac{f\left(s\right)m\left(\vartheta
  ,s\right)}{\varepsilon ^2\sigma \left(s\right)^2}\;{\rm
  d}X_s-\int_{0}^{t}\frac{f\left(s\right)^2m\left(\vartheta ,s\right)^2}{2\varepsilon ^2\sigma
  \left(s\right)^2}\;{\rm d}s\right) ,\qquad \vartheta \in\Theta  . 
\end{align}
Hence the Fisher information 
\begin{align*}
{\rm I}^t\left(\vartheta \right)=\varepsilon ^2 \Ex_\vartheta  \left[\frac{\partial \ln
    L\left(\vartheta ,X^t\right)}{\partial \vartheta }\right]^2 =
\int_{0}^{t}\frac{f\left(s\right)^2\Phi 
  \left(0,s\right)^2}{\sigma \left(s\right)^2}{\rm d}s 
\end{align*}
does not depend on $\vartheta $  and the MLE has an explicit expression
\begin{align}
\label{21-z4}
\hat\vartheta _{t,\varepsilon }=\left(\int_{0}^{t}\frac{f\left(s\right)^2\Phi
  \left(0,s\right)^2}{\sigma \left(s\right)^2}{\rm d}s  \right)^{-1} \;\int_{0}^{t}\frac{f\left(s\right)\Phi
  \left(0,s\right)}{\sigma \left(s\right)^2}\left[{\rm
    d}X_s-f\left(s\right)H\left(s,X^s\right) {\rm d}s\right]. 
\end{align}
We write  ${\rm I}^t\left(\vartheta \right) $ to mention the  estimated
parameter and  $\vartheta _0$ for the true value of $\vartheta $. 
It is easy to see that
\begin{align*}
\varepsilon ^{-1}\left(\hat\vartheta _{t,\varepsilon }-\vartheta _0\right)=
\left(\int_{0}^{t}\frac{f\left(s\right)^2\Phi \left(0,s\right)^2}{\sigma
  \left(s\right)^2}{\rm d}s \right)^{-1}
\;\int_{0}^{t}\frac{f\left(s\right)\Phi \left(0,s\right)}{\sigma
  \left(s\right)}{\rm d}\bar W_s\;\sim \;{\cal N}\left(0,{\rm
  I}^t\left(\vartheta _0\right)^{-1}\right). 
\end{align*}
Here we used the representation of observations
\begin{align*}
{\rm d}X_s=f\left(s\right)m\left(\vartheta _0,s\right){\rm d}s+\varepsilon
\sigma \left(s\right) {\rm d}\bar W_s ,\qquad X_0=0,\quad 0\leq s\leq t
\end{align*}
where $\bar W_s,0\leq s\leq t $ is an innovation Wiener process (see Theorem
7.6 in \cite{LSv1}). 

The estimator $\hat\vartheta _{t,\varepsilon } $ can be written in the
recurrent form as follows:
\begin{align*}
\hat\vartheta _{t,\varepsilon }=\frac{S\left(t\right)}{{\rm I}^t}
\end{align*}
where 
\begin{align*}
{\rm d}S\left(t\right)&=\frac{f\left(t\right)\Phi \left(0,t\right)}{\sigma
  \left(t\right)^2}\left[{\rm d}X_t-f\left(t\right)H\left(t,X^t\right) {\rm
    d}t\right],\qquad S\left(0\right)=0,\\
{\rm d}{\rm I}^t&=\frac{f\left(t\right)^2\Phi
  \left(0,t\right)^2}{\sigma \left(t\right)^2}{\rm d}t,\qquad {\rm I}^0=0,\\
{\rm d}H\left(t,X^t\right)&=A\left(t\right)H\left(t,X^t\right){\rm d}t+
{f\left(t\right)\gamma_* \left(t\right)}{\sigma   \left(t\right)^{-2}}{\rm d}X_t.
\end{align*}
It is possible to write the stochastic differential of the process
$\hat\vartheta _{t,\varepsilon },0\leq t\leq T $. 
 Let us multiply  both sides of \eqref{21-z4} by
${\rm I}^t\left(\vartheta \right) $, denote
$g\left(t\right)={f\left(t\right)\Phi \left(0,t\right)}{\sigma
  \left(t\right)^{-1}} $ and take the differential, then we obtain  
\begin{align*}
\hat\vartheta _{t,\varepsilon }\frac{f\left(t\right)^2\Phi
  \left(0,t\right)^2}{\sigma \left(t\right)^2} {\rm d}t+{\rm
  I}^t\left(\vartheta \right){\rm d}\hat\vartheta _{t,\varepsilon }=
\frac{f\left(t\right)\Phi \left(0,t\right)}{\sigma \left(t\right)^2}\left[ {\rm
  d}X_t-f\left(t\right)H\left(t,X^t\right) {\rm d}t\right]
\end{align*}
and
\begin{align}
\label{21-z4a}
{\rm d}\hat\vartheta _{t,\varepsilon }=-\frac{g\left(t\right)^2}{{\rm
    I}^t\left(\vartheta \right)}\hat\vartheta _{t,\varepsilon }{\rm d}t+ 
\frac{g\left(t\right)}{{\rm I}^t\left(\vartheta
  \right)\sigma \left(t\right)}\left[ {\rm
    d}X_t-f\left(t\right)H\left(t,X^t\right) {\rm     d}t\right].
\end{align}
In these two representations there is a question of the initial value because
$\hat\vartheta _{0,\varepsilon }=0/0 $. To avoid this problem we can  fix some
(small) $\tau $ and calculate the preliminary estimator
\begin{align}
\label{21-z5}
\hat\vartheta _{\tau ,\varepsilon }=\frac{1}{{\rm I}^\tau }\;\int_{0}^{\tau
}\frac{g\left(s\right)}{\sigma \left(s\right)}\left[{\rm 
    d}X_s-f\left(s\right)H\left(s,X^s\right) {\rm d}s\right] 
\end{align} 
and then  use the equation \eqref{21-z4a} to compute $\hat\vartheta _{t
  ,\varepsilon },\tau\leq t\leq T $.

Let us verify that \eqref{21-z4a}-\eqref{21-z5} gives the same estimator, i.e.,
that $\hat\vartheta _{\tau ,\varepsilon }\sim {\cal N}\left(\vartheta
_0,\varepsilon ^2{\rm I}^t\left(\vartheta _0\right)^{-1} \right) $.  It is
sufficient to calculate $\Ex_{\vartheta _0}\left(\hat\vartheta _{\tau
  ,\varepsilon }-\vartheta _0 \right)^2=\varepsilon ^2{\rm I}^t\left(\vartheta
_0\right)^{-1}$.  We have
\begin{align*}
\hat\vartheta _{t ,\varepsilon }&=\hat\vartheta _{\tau ,\varepsilon }
\tilde\Phi \left(\tau ,t\right)+\int_{\tau }^{t} \tilde\Phi \left(s
,t\right)\frac{g\left(s\right)}{{\rm I}^s\sigma \left(s\right)} \left[{\rm
    d}X_s-f\left(s\right)H\left(s,X^s\right){\rm d}s \right] ,
\end{align*}
where
\begin{align*}
\tilde\Phi \left(s ,t\right)=\exp\left(-\int_{s}^{t}\frac{g\left(r\right)^2}{{\rm
    I}^r}{\rm d}r\right) .
\end{align*}
Therefore the substitution of the observations gives the expression
\begin{align*}
\hat\vartheta _{t ,\varepsilon }&=\hat\vartheta _{\tau ,\varepsilon }
\tilde\Phi \left(\tau ,t\right)+\vartheta _0\int_{\tau }^{t} \tilde\Phi
\left(s ,t\right)\frac{g\left(s\right)^2}{{\rm I}^s}{\rm
  d}s+\varepsilon\int_{\tau }^{t} \tilde\Phi \left(s ,t\right)
\frac{g\left(s\right)}{{\rm I}^s} {\rm d}\bar W_s.
\end{align*}
Remark that
\begin{align*}
\int_{s}^{t}\frac{g\left(r\right)^2}{{\rm I}^r}{\rm
  d}r=\int_{s}^{t}\frac{1}{{\rm I}^r}{\rm d}\left(\int_{0 }^{r} g\left(q\right)^2
    {\rm d}q\right)=\int_{s}^{t}\frac{1}{{\rm I}^r}{\rm d} {\rm I}^r=\ln {\rm I}^t-\ln {\rm I}^s
\end{align*}
and
\begin{align*}
 \tilde\Phi \left(s ,t\right)=\frac{{\rm I}^s}{{\rm I}^t}.
\end{align*}
Therefore
\begin{align*}
\int_{\tau }^{t} \tilde\Phi \left(s ,t\right)\frac{g\left(s\right)^2}{{\rm
    I}^s}{\rm d}s=\frac{1}{{\rm I}^t}\int_{\tau }^{t}g\left(s\right)^2{\rm
  d}s=\frac{{\rm I}^t-{\rm I}^\tau }{{\rm I}^t} 
\end{align*}
and we obtain the relation
\begin{align*}
\hat\vartheta _{t ,\varepsilon }&=\hat\vartheta _{\tau ,\varepsilon }\frac{{\rm I}^\tau}{{\rm I}^ t}
+\vartheta _0  \frac{{\rm I}^t-{\rm I}^\tau }{{\rm I}^t}
+\frac{\varepsilon}{{\rm I}^t }\int_{\tau }^{t} 
g\left(s\right) {\rm d}\bar W_s.
\end{align*}
Hence
\begin{align*}
\varepsilon^{-1}\left(\hat\vartheta _{t ,\varepsilon }-\vartheta _0\right)=\varepsilon^{-1}\left(\hat\vartheta _{\tau
  ,\varepsilon }-\vartheta _0\right)\frac{{\rm I}^\tau}{{\rm I}^ t} +\frac{1}{{\rm I}^t }\int_{\tau }^{t} 
g\left(s\right) {\rm d}\bar W_s
\end{align*}
and
\begin{align*}
\varepsilon^{-2}\Ex_{\vartheta _0}\left(\hat\vartheta _{t ,\varepsilon
}-\vartheta _0\right)^2& =\varepsilon^{-2}\Ex_{\vartheta _0}\left(\hat\vartheta _{\tau  ,\varepsilon
}-\vartheta _0\right)^2\left(\frac{{\rm I}^\tau}{{\rm I}^ t}\right)^2 +\frac{1}{\left({\rm I}^t \right)^2}\int_{\tau }^{t}
g\left(s\right)^2{\rm d}s\\
& =\frac{{\rm I}^\tau}{\left({\rm I}^t\right)^2 }+
\frac{{\rm I}^t-{\rm      I}^\tau  }{\left({\rm I}^t\right)^2 } =\frac{1}{{\rm I}^ t}.
\end{align*}
Therefore $\varepsilon^{-1}\left(\hat\vartheta _{t ,\varepsilon
}-\vartheta _0\right)\sim {\cal N}\left(0,{\rm I}^t\left(\vartheta _0\right)^{-1}\right) $.

We have the recurrent equations for the estimator of the parameter and the
adaptive Kalman filter is defined in two steps. First we introduce the random
process 
\begin{align}
\label{21-z6}
{\rm d}n^*_t&=\left[a\left(t\right)-\frac{f\left(t\right)^2\gamma
    _*\left(t\right)}{\sigma \left(t\right)^2} \right]n^*_t{\rm
  d}t+\frac{f\left(t\right)\gamma_* \left(t\right)}{\sigma
  \left(t\right)^2}{\rm d}X_t
,\qquad \;n^*_0 =0.
\end{align}
The solution of this equation is (see \eqref{21-z3})
\begin{align*}
n^*_t&=\int_{0}^{t}\Phi \left(s,t\right)f\left(s\right)\gamma
_*\left(s\right)\sigma \left(s\right)^{- 2}{\rm d}X_s.
\end{align*}

The adaptive Kalman filter defined by the equation
\begin{align}
\label{21-z7}
m^\star_t=\hat\vartheta _{t ,\varepsilon}\Phi \left(0,t\right)+n^*_t.
\end{align}
The error of estimation is
\begin{align*}
\varepsilon ^{-1}\left(m^\star_t-m\left(\vartheta
_0,t\right)\right)&=\varepsilon ^{-1}(\hat\vartheta _{t 
  ,\varepsilon}-\vartheta _0)\Phi \left(0,t\right)=\frac{\Phi
  \left(0,t\right)}{{\rm I}^t} \int_{0}^{t} \frac{f\left(s\right)
  \gamma _*\left(s\right)}{\sigma \left(s\right)}{\rm d} \bar W_s
\end{align*}
The lower bound in this problem is
\begin{align*}
\lim_{\nu \rightarrow 0}\Liminf_{\varepsilon \rightarrow 0}\sup_{\left|\vartheta
  -\vartheta _0\right|\leq \nu } \varepsilon ^{-2}\Ex_\vartheta \left[\bar
  m_{t,\varepsilon }-m\left(\vartheta ,t\right)\right]^2 \geq \frac{\Phi
  \left(0,t\right)^2}{{\rm I}^t\left(\vartheta _0\right)}. 
\end{align*}
To prove this bound we follow, e.g.,  \cite{Kut24b}. We have an elementary
inequality
\begin{align*}
\sup_{\left|\vartheta -\vartheta _0\right|\leq \nu } \Ex_\vartheta \left[\bar
  m_{t,\varepsilon }-m\left(\vartheta ,t\right)\right]^2 \geq \int_{\vartheta
  _0-\nu }^{\vartheta _0+\nu }\Ex_\vartheta \left[\bar m_{t,\varepsilon
  }-m\left(\vartheta ,t\right)\right]^2p_\nu \left(\vartheta \right){\rm
  d}\vartheta
\end{align*}
where we introduced the continuous positive density $p_\nu \left(\vartheta
\right),\left(\vartheta _0-\nu ,\vartheta _0+\nu\right) $. Denote
$\tilde m\left(t\right) $ the Bayes estimator. Then by the definition of
Bayes estimator 
\begin{align*}
\int_{\vartheta _0-\nu }^{\vartheta _0+\nu }\Ex_\vartheta \left[\bar
  m_{t,\varepsilon }-m\left(\vartheta ,t\right)\right]^2p_\nu \left(\vartheta
\right){\rm d}\vartheta \geq \int_{\vartheta _0-\nu }^{\vartheta _0+\nu
}\Ex_\vartheta \left[\tilde m\left(t\right)- m\left(\vartheta
  ,t\right)\right]^2p_\nu \left(\vartheta \right){\rm d}\vartheta .
\end{align*}
Recall that
\begin{align*}
\tilde m\left(t\right)&=\Ex \left(m\left(\vartheta
,t\right)|{\goth F}_t^X\right)=\int_{\alpha }^{\beta }m\left(\theta
,t\right)p_\nu\left(\theta |X^t\right){\rm d}\theta=\frac{\ds\int_{\alpha }^{\beta }m\left(\theta
,t\right)p_\nu\left(\theta \right)L\left(\theta,X^t\right){\rm
    d}\theta}{\ds\int_{\alpha }^{\beta }p_\nu\left(\theta
  \right)L\left(\theta,X^t\right){\rm d}\theta}\\
&=H\left(t,X^t\right)+\Phi \left(0,t\right)\frac{\ds\int_{\alpha }^{\beta }\theta
p_\nu\left(\theta\right)L\left(\theta,X^t\right){\rm
    d}\theta}{\ds\int_{\alpha }^{\beta }p_\nu\left(\theta
  \right)L\left(\theta,X^t\right){\rm d}\theta}.
\end{align*}
If we change the variable $\theta =\vartheta +\varepsilon u$ and denote $
\theta _u=\vartheta +\varepsilon u$,
$Z_\varepsilon \left(\vartheta ,u,t\right)=L\left(\theta _u,X^t\right)/ L\left(\vartheta
,X^t\right)$, $\UU_\varepsilon =\left[\varepsilon ^{-1}\left(\alpha
  -\vartheta\right),\varepsilon ^{-1}\left(\beta -\vartheta \right)\right]$
then 
\begin{align*}
\frac{\tilde
m\left(t\right)-m\left(\vartheta
,t\right)}{\varepsilon \Phi \left(0,t\right)   }&=\frac{\ds\int_{\UU_\varepsilon }^{ }u
  p_\nu\left(\theta_u\right)Z_\varepsilon \left(\vartheta ,u,t\right){\rm
    d}u}{\ds\int_{\UU_\varepsilon }^{ } p_\nu\left(\theta_u\right)Z_\varepsilon
  \left(\vartheta ,u,t\right){\rm d}u}=\frac{\ds\int_{\UU_\varepsilon }^{ }u
  Z_\varepsilon \left(\vartheta ,u,t\right){\rm
    d}u}{\ds\int_{\UU_\varepsilon }^{ } Z_\varepsilon
  \left(\vartheta ,u,t\right){\rm d}u}\left(1+o\left(1\right)\right)\\
&\Longrightarrow \frac{\ds\int_{\cal R }^{ }u
  Z \left(\vartheta ,u,t\right){\rm
    d}u}{\ds\int_{\cal R }^{ } Z
  \left(\vartheta ,u,t\right){\rm d}u}={\rm I}^t\left(\vartheta
_0\right)^{-1}\int_{0}^{t}\frac{f\left(s\right)\Phi \left(0,s\right)}{\sigma
  \left(s\right)} {\rm d} W\left(s\right)
\end{align*}
Here we used the continuity of $p_\nu\left(\cdot \right)$ and denoted
$W\left(s\right),0\leq s\leq t $ a Wiener process. It can be shown that the
moments of BE converge too (see \cite{IH81}).

Therefore
\begin{align*}
\lim_{\nu \rightarrow 0}\Liminf_{\varepsilon \rightarrow 0}\sup_{\left|\vartheta
  -\vartheta _0\right|\leq \nu } \varepsilon ^{-2}\Ex_\vartheta \left[\bar
  m_{t,\varepsilon }-m\left(\vartheta ,t\right)\right]^2 \geq \frac{\Phi
  \left(0,t\right) ^2 }{{\rm I}^t\left(\vartheta _0\right)}.
\end{align*}

As 
\begin{align*}
 \varepsilon ^{-2}\Ex_\vartheta \left[m^\star_t-m\left(\vartheta ,t\right)\right]^2 = \frac{\Phi
  \left(0,t\right)^2}{{\rm I}^t\left(\vartheta \right)}
\end{align*}
the  approximation $m^\star_t $ is asymptotically efficient.

\subsubsection{Unknown initial value and parameter}

Suppose that the partially observed system is
\begin{align}
\label{21-z8}
{\rm d}X_t&=f\left(t\right)Y_t\,{\rm d}t+\varepsilon \sigma\left(t\right) \,{\rm d}W_t,\qquad X_0=0,\qquad
0\leq t\leq T,\\
\label{21-z9}
{\rm d}Y_t&=a\left(\theta_2,t\right) Y_t\,{\rm d}t+\varepsilon
b\left(t\right)\,{\rm d}V_t,\qquad \quad Y_0=\theta_1.
\end{align}
The unknown parameter is $\vartheta =\left(\theta _1,\theta _2\right)\in\Theta
=\left(\alpha _1,\beta _1\right)\times\left(\alpha _2,\beta _2\right)$,
$\alpha _1>0$.  The functions $f\left(\cdot \right),\sigma \left(\cdot
\right),a\left(\cdot \right),b\left(\cdot \right) $ satisfy the condition
${\cal A}_*$. Our goal is to find an approximation of the conditional
expectation $m\left(\vartheta ,t\right)=\Ex_\vartheta \left(Y_t|{\goth
  F}_t^X\right)$ and describe its properties  in the asymptotic $\varepsilon
\rightarrow 0$. 

Remark that the LR formula for $L\left(\vartheta ,X^t\right)$ is
\eqref{21-LR}, where the function $m\left(\vartheta ,s\right),0\leq s\leq t$
is solution of the equations like \eqref{21-z1},\eqref{21-z2}, where
$a\left(t\right)=a\left(\theta _1,t\right)$. Therefore the calculation of the MLE $\hat\vartheta
_{\varepsilon ,t}$ defined by the equation
\begin{align*}
L\left(\hat\vartheta
_{\varepsilon ,t} ,X^t\right)=\sup_{\vartheta \in\Theta }L\left(\vartheta ,X^t\right)
\end{align*}
is a very difficult computational  problem. It requires the solutions of
filtration equations for all or many $\vartheta \in\Theta $ and for all
$t\in\left[0,T\right]$. That is why we propose below the construction of the
estimator-process which needs such solutions just for one value of
$\vartheta$. Moreover this estimator has recurrent expression which reduces
one more the computation  of the adaptive filter.

 Introduce conditions ${\cal B}$.
{\it
\begin{description}
\item[${\cal B}_1$.]  The function $a\left(\vartheta ,\cdot
\right),\vartheta \in \bar\Theta $ has two continuous derivatives on
$\vartheta $. 

\item[${\cal B}_2$.]  The functions $f\left(\cdot \right),\sigma \left(\cdot
\right),a\left(\cdot \right),b\left(\cdot \right) $ have continuous
derivatives in $t\in\left[0,T\right]$.
\end{description}
}

The adaptive recurrent K-B filter is constructed in several steps. First we
estimate the parameter $\theta _2$ by the observations on the small
interval $\left[0,\tau _\varepsilon \right]$, where $\tau _\varepsilon
\rightarrow 0$. Then we use this estimator in the construction of the
parameter $\theta _1$ by observations on the interval $\left[0,\tau
  \right]$. The obtained estimator of $\vartheta $ is used to introduce the
One-step MLE-process. The last step is to substitute this One-step
estimator-process in the equations of Kalman-Bucy filtration. 

The limit system $\left(\varepsilon =0\right)$ is, of course,
\begin{align*}
\frac{{\rm d}x_t\left(\vartheta_0 \right)}{{\rm
    d}t}&=f\left(t\right)y_t\left(\vartheta _0\right), \qquad\qquad 
x_0\left(\vartheta_0 \right) =0,\qquad \quad 0\leq t\leq T,\\
\frac{{\rm d}y_t\left(\vartheta_0 \right)}{{\rm
    d}t}&=a\left(\theta _{2,0},t\right)y_t\left(\vartheta _0\right), \qquad
y_0\left(\vartheta_0 \right) =\theta _{1,0}.
\end{align*}
Here $\vartheta _0=\left(\theta _{1,0},\theta _{2,0}\right)$ is the true value
of $\vartheta $. 

 Note that the limit  function of the process $Y_t$  is 
\begin{align*}
y_t\left(\vartheta_0 \right)=\theta_{1,0} \;\phi
\left(\theta_{2,0},0,t\right),\qquad \qquad 
\phi
\left(\theta_{2,0},s,t\right)=\exp\left(\int_{s}^{t}a\left(\theta_{2,0},v\right){\rm
  d}v\right)  
\end{align*}
and 
\begin{align*}
Y_t&=\theta_{1,0} \;\phi \left(\theta_{2,0},0,t\right)+\varepsilon \int_{0}^{t}\phi
\left(\theta_{2,0},s,t\right)b\left(s\right){\rm d}V_s .
\end{align*}
The observations have the following representation as $\tau _\varepsilon \rightarrow 0$
\begin{align*}
X_{\tau_\varepsilon } &=\theta_{1,0}\int_{0}^{\tau_\varepsilon
}f\left(t\right)\phi\left(\theta_{2,0},0,t\right)\;{\rm d}t +\varepsilon
\int_{0}^{\tau_\varepsilon }\sigma \left(s\right){\rm d} W_s\\ 
&\qquad+\varepsilon\int_{0}^{\tau_\varepsilon }
f\left(t\right)\int_{0}^{t}\phi\left(\theta_{2,0},s,t\right)b\left(s\right){\rm
  d}V_s\;{\rm d}t\\
 &=\theta_{1,0}f\left(0\right)\tau_\varepsilon+\varepsilon
\sigma \left(0\right) W_{\tau _\varepsilon }+\varepsilon \tau _\varepsilon
^{3/2} f\left(0\right)b\left(0\right)\int_{0}^{1}r{\rm d}v_{r,\varepsilon }+ O\left(\varepsilon \tau _\varepsilon ^{3/2}\right),
\end{align*}
where we used the Taylor formula and introduced $v_{r,\varepsilon },0\leq
r\leq 1$ as a Wiener process independent of $ W_{\tau _\varepsilon } $. Here
and below if the object is random then convergences in $O\left(\cdot \right)$
and $o\left(\cdot \right)$ are understood {\it in probability}, for example,
$\tau _\varepsilon ^{-1}o\left(\tau _\varepsilon \right)\rightarrow 0$ in
probability.

Hence the method of moments estimator (MME) can be 
\begin{align*}
\theta ^*_{1,\tau _\varepsilon }=\frac{X_{\tau_\varepsilon
}}{f\left(0\right)\tau_\varepsilon}=\theta_{1,0}+\frac{\varepsilon
}{\sqrt{\tau _\varepsilon }}\sigma \left(0\right)\xi _{1,\varepsilon
}+\tau _\varepsilon O\left(1 \right)+\varepsilon \tau _\varepsilon ^{1/2}  b\left(0\right)\int_{0}^{1}r{\rm d}v_{r,\varepsilon }
+\varepsilon\tau _\varepsilon O\left(1 \right).
\end{align*}
The solution of the equation $\varepsilon /\sqrt{\tau _\varepsilon }=\tau
_\varepsilon  $ is the rate $\tau _\varepsilon =\varepsilon^{2/3}
$. Therefore, we obtain asymptotic normality 
\begin{align*}
\varepsilon ^{-2/3}\left(\theta ^*_{1,\tau _\varepsilon }-\theta_{1,0}
\right)\Longrightarrow {\cal N}\left(q,d^2\right)
\end{align*}
with some parameters $q,d^2$. The particular values of these parameters are
not interesting. Moreover, we have the convergence of all polynomial  moments too.

Let us fix some (small) $\tau >0$ and define  the random process 
\begin{align*}
x_{t,\varepsilon }\left(\theta _2\right)=\theta ^*_{1,\tau _\varepsilon}\int_{0}^{t
}f\left(s\right)\phi \left(\theta _2,0,s\right){\rm d}s ,\qquad \quad 0\leq t\leq \tau,
\end{align*}
and define the minimum distance estimator (MDE) $\check\theta_{2,\tau  }
$ by the equation
\begin{align*}
\int_{0}^{\tau }\left[X_t-x_{t,\varepsilon }\left(\check\theta_{2,\tau 
  }\right)\right]^2{\rm d}t=\inf_{\alpha _2<\theta _{2}<\beta _2}\int_{0}^{\tau
}\left[X_t-x_{t,\varepsilon }\left(\theta _2\right)\right]^2{\rm d}t.
\end{align*}

 Introduce the identifiability condition

\medskip

${\cal B}_3$. {\it For a   given $\tau $,  any $\nu >0$ and any $\theta
  _{2,0}\in \left[\alpha _2,\beta _2\right] $ }
\begin{align*}
g\left(\vartheta _0,\nu \right)=\inf_{\left|\theta _{2}-\theta _{2,0}\right|\geq \nu  }\int_{0}^{\tau
}\left[F\left(\theta_2,t\right)-F\left(\theta_{2,0},t\right)\right]^2{\rm
  d}t >0.
\end{align*}
Here 
\begin{align*}
F\left(\theta,t\right)=\int_{0}^{t }f\left(s\right)\phi \left(\theta
,0,s\right){\rm d}s.
\end{align*}
We have
\begin{align*}
X_t-x_{t,\varepsilon }\left(\check\theta_{2,\tau 
}\right)&=x_t\left(\vartheta _0\right)-x_{t,\varepsilon
}\left(\check\theta_{2,\tau  }\right)  +\varepsilon \int_{0}^{t}\sigma
\left(s\right){\rm d}W_s\\
&\qquad +\varepsilon\int_{0}^{t}f\left(s\right)\int_{0}^{s}\phi
\left(\theta _{2,0},r,s\right)b\left(r\right){\rm d}V_r\,{\rm d}s\\ 
&=\left(\theta _{1,0}-\theta ^*_{1,\tau_\varepsilon  }
\right)F\left(\theta_{2,0},t\right)+\theta ^*_{1,\tau_\varepsilon  }
\left[F\left(\theta_{2,0},t\right)-F\left(\check\theta_{2,\tau 
  },t\right) \right]+\varepsilon  \;\eta _t,
\end{align*}
where $\eta _t,0\leq t\leq \tau $ is a Gaussian process with bounded
variance defined by the last equality. 

Below $\left\|h_t\right\|_\tau $ is   ${\cal L}_2\left[0,\tau
  \right]$-norm. To prove the consistency of the estimator $\check\theta_{2,\tau 
  }$  we write
\begin{align*}
&\Pb_{\vartheta _0}\left(\left|\check\theta_{2,\tau }-\theta_{2,0
}\right|\geq \nu \right)=\Pb_{\vartheta _0}\left(\inf_{\left|\theta
    _{2}-\theta _{2,0}\right|\leq \nu }\left\|X_t-x_{t,\varepsilon
  }\left(\theta _{2}\right)\right\|_\tau > \inf_{\left|\theta _{2}-\theta
    _{2,0}\right|\geq \nu }\left\|X_t-x_{t,\varepsilon }\left(\theta
  _{2}\right)\right\|_\tau\right)\\
&\qquad \leq \Pb_{\vartheta _0}\left(\inf_{\left|\theta
    _{2}-\theta _{2,0}\right|\leq \nu }\left(\left\|X_t-x_{t
  }\left(\vartheta _{0}\right)\right\|_\tau+\left\|x_{t
  }\left(\vartheta _{0}\right)-x_{t,\varepsilon
  }\left(\theta _{2,0}\right)\right\|_\tau+\left\|x_{t,\varepsilon
  }\left(\theta _{2}\right)-x_{t,\varepsilon
  }\left(\theta _{2,0}\right)\right\|_\tau \right)  \right.\\
&\qquad\qquad  \left.> \inf_{\left|\theta _{2}-\theta
    _{2,0}\right|\geq \nu }\left\|x_{t,\varepsilon
  }\left(\theta _{2}\right)-x_{t,\varepsilon
  }\left(\theta _{2,0}\right)\right\|_\tau -\left\|X_t-x_{t}\left(\vartheta
  _{0}\right)\right\|_\tau -\left\|x_{t
  }\left(\vartheta _{0}\right)-x_{t,\varepsilon
  }\left(\theta _{2,0}\right)\right\|_\tau\right)\\
&\qquad \leq \Pb_{\vartheta _0}\left(2\left\|X_t-x_{t }\left(\vartheta
  _{0}\right)\right\|_\tau+2\left\|x_{t }\left(\vartheta
  _{0}\right)-x_{t,\varepsilon }\left(\theta _{2,0}\right)\right\|_\tau >
  \alpha _1\,g\left(\theta_{2,0},\nu \right)\right).
\end{align*}
By Tchebyshev inequality we have
\begin{align*}
\Pb_{\vartheta _0}\left(\left|\check\theta_{2,\tau }-\theta_{2,0
}\right|\geq \nu \right)&\leq \frac{8}{\alpha _1^2g\left(\vartheta _0,\nu
  \right)^2}\left(\Ex_{\vartheta _0}\left\|X_t-x_{t }\left(\vartheta
  _{0}\right)\right\|_\tau^2+\Ex_{\vartheta _0}\left\|x_{t }\left(\vartheta
  _{0}\right)-x_{t,\varepsilon }\left(\theta _{2,0}\right)\right\|_\tau^2
\right)\\
&\leq \frac{C\varepsilon ^2+C\varepsilon ^{4/3}}{\alpha _1^2g\left(\vartheta
  _0,\nu   \right)^2} \longrightarrow 0.
\end{align*}
Therefore the estimator $\bar\vartheta _{\tau ,\varepsilon }=\left(\theta
_{1,\tau _\varepsilon },\theta _{2,\tau } \right)$ is consistent. It can be
shown that the difference $\varepsilon ^{-2/3}\left(\bar\vartheta _{\tau
  ,\varepsilon }-\vartheta _0 \right)$ is asymptotically normal and the
moments $\varepsilon ^{-2p/3}\Ex_{\vartheta _0} \left\|\bar\vartheta _{\tau
  ,\varepsilon }-\vartheta _0\right\|^p  $ are bounded for any $p>0$. This
estimator will be used for the construction of One-step MLE-process. We need
the equations of Kalman-Bucy filter for this model of observations:
\begin{align}
\label{21-z13}
&{\rm d}m\left(\vartheta ,t\right)=A\left(\theta _2,t\right) m\left(\vartheta
,t\right){\rm d}t+\frac{f\left(t\right)\gamma_* \left(\theta
  _2,t\right)}{\sigma \left(t\right)^2}\;{\rm    d}X_t,\qquad \quad m\left(\vartheta ,0\right)=\theta _1,\\
&\frac{{\rm d}\gamma_* \left(\theta_2,t\right) }{{\rm d}t}=2a\left(\theta
_2,t\right)\gamma_* \left(\theta_2,t\right) - \frac{f\left(t\right)^2\gamma_*
  \left(\theta_2,t\right) ^2}{\sigma \left(t\right)^2}+b\left(t\right)^2,\qquad \gamma_*
 \left(\theta_2,0\right) =0,
 \label{21-z14}
\end{align}
where
\begin{align*}
A\left(\theta _2,t\right)=a\left(\theta _2,t\right) -
\frac{f\left(t\right)^2\gamma_* \left(\theta _2,t\right)}{\sigma
  \left(t\right)^2}.
\end{align*}
 
The derivatives $\dot m_1\left(\vartheta ,t\right)=\partial m\left(\vartheta
,t\right)/\partial \theta _1$ and  $\dot m_2\left(\vartheta ,t\right)=\partial m\left(\vartheta
,t\right)/\partial \theta _2$ satisfy the equations
\begin{align*}
\frac{{\rm d}\dot m_1\left(\vartheta ,t\right)}{{\rm d}t}&=A\left(\theta _2,t\right) \dot m_1\left(\vartheta
,t\right),\qquad \qquad  \dot m_1\left(\vartheta ,t\right)=1,\\
{\rm d}\dot m_2\left(\vartheta ,t\right)&=A\left(\theta _2,t\right) \dot m_2\left(\vartheta
,t\right){\rm d}t+ \dot A_2\left(\theta _2,t\right) m\left(\vartheta
,t\right){\rm d}t+               \frac{f\left(t\right)\dot\gamma_* \left(\theta
  _2,t\right)}{\sigma \left(t\right)^2}\;{\rm    d}X_t,\; \dot m_2\left(\vartheta ,0\right)=0.
\end{align*}
Therefore 
\begin{align*}
\dot m_1\left(\vartheta ,t\right)&=\Phi \left(\theta _2,0,t\right),\qquad \quad \Phi
\left(\theta _2,s,t\right)=\exp\left(\int_{s}^{t}A\left(\theta _2,r\right)
     {\rm d}r\right),\\
 \dot m_2\left(\vartheta ,t\right)&=\int_{0}^{t}\Phi
\left(\theta _2,s,t\right)\dot A_2\left(\theta _2,s\right) m\left(\vartheta
,s\right){\rm d}s+\int_{0}^{t}\Phi
\left(\theta _2,s,t\right) \frac{f\left(s\right)\dot\gamma_* \left(\theta
  _2,s\right)}{\sigma \left(s\right)^2}\;{\rm    d}X_s.
\end{align*}
The limit ($\varepsilon =0$) values are: for $\dot m_1\left(\vartheta
,t\right)=\Phi \left(\theta _2,0,t\right)\equiv     \dot y_1\left(\vartheta ,t\right) $ and for $ \dot
m_2\left(\vartheta ,t\right) $ we obtain the expression
\begin{align*}
\dot y_2\left(\vartheta ,\vartheta _0,t\right)&=\int_{0}^{t}\Phi \left(\theta
_2,s,t\right)\dot A_2\left(\theta _2,s\right)y_s\left(\vartheta ,\vartheta
_0\right){\rm d}s\\
&\qquad \qquad +\int_{0}^{t}\Phi \left(\theta _2,s,t\right)
\frac{f\left(s\right)^2\dot\gamma_* \left(\theta _2,s\right)}{\sigma
  \left(s\right)^2}y_s\left(\vartheta _0\right)  \;{\rm d}s
\end{align*}
where
\begin{align*}
y_t\left(\vartheta _0\right) &=\theta _{1,0}\phi \left(\theta
_{2,0},0,t\right) ,\qquad \quad \phi \left(\theta
_{2,0},0,t\right)=\exp\left(\int_{0}^{t}a\left(\theta _{2,0},s\right){\rm
  d}s\right),\\
y_t\left(\vartheta ,\vartheta _0\right)&=\theta _1\Phi
\left(\theta _2,0,t\right)+\theta _{1,0}\int_{0}^{t}\Phi
\left(\theta _2,s,t\right)\frac{f\left(s\right)^2\dot\gamma_* \left(\theta
  _2,s\right)}{\sigma \left(s\right)^2}\phi
\left(\theta _{2,0},0,s\right){\rm d}s.
\end{align*}
The function $\dot y_2\left(\vartheta ,\vartheta _0,t\right) $ at the point
$\vartheta =\vartheta _0$ takes the value $\dot y_2\left(\vartheta _0,t\right)
$ where
\begin{align*}
\dot y_2\left(\vartheta _0,t\right)&=\int_{0}^{t}\Phi \left(\theta
_{2,0},s,t\right)\dot a\left(\theta _{2,0},s\right)y_s\left( \vartheta
_0\right){\rm d}s \\
 &=\theta _{1,0}\int_{0}^{t}\Phi \left(\theta
_{2,0},s,t\right)\phi \left(\theta _{2,0},0,s\right) \dot a\left(\theta
_{2,0},s\right){\rm d}s.
\end{align*}

The Fisher information matrix is
\begin{align*}
{\bf I}_\tau ^t\left(\vartheta \right)=\left(
\begin{array}{cc}
\ds\int_{\tau }^{t}\frac{f\left(s\right)^2\dot y_1\left(\vartheta ,s\right)^2}{\sigma \left(s\right)^2}{\rm
  d}s&\ds \int_{\tau }^{t}\frac{f\left(s\right)^2\dot y_1\left(\vartheta
,s\right)\dot y_2\left(\vartheta ,s\right)}{\sigma \left(s\right)^2}{\rm d}s \\ 
\ds\int_{\tau
}^{t}\frac{f\left(s\right)^2\dot y_1\left(\vartheta ,s\right)\dot y_2\left(\vartheta ,s\right)}{\sigma \left(s\right)^2}{\rm d}s&\ds\int_{\tau }^{t}\frac{f\left(s\right)^2\dot
y_2\left(\vartheta ,s\right)^2}{\sigma \left(s\right)^2}{\rm d}s\\
\end{array}
\right).
\end{align*}

Introduce  the condition 

\medskip

${\cal B}_3$. {\it For a given $\tau $ and any $t\in (\tau ,T]$ and any
$\vartheta \in\Theta $ the Fisher information is non degenerate }
\begin{align*}
\inf_{\left\|{\rm e} \right\|=1,  {\rm e}\in {\cal R}^2}\;\;{\rm e}^\top{\bf
  I}_\tau ^t\left(\vartheta \right){\rm e}>0. 
\end{align*}

The One-step MLE-process is
\begin{align}
\label{21-z15}
\vartheta _{t,\varepsilon }^\star=\bar\vartheta _{\tau ,\varepsilon }+{\bf
  I}_\tau ^t\left(\bar\vartheta _{\tau ,\varepsilon } \right)^{-1} \int_{\tau
}^{t}\frac{f\left(s\right)\dot {\rm m}\left(\bar\vartheta _{\tau ,\varepsilon
  },s\right)}{\sigma \left(s\right)^2 }\left[{\rm
    d}X_s-f\left(s\right)m\left(\bar\vartheta _{\tau ,\varepsilon },s
  \right){\rm d}s\right].
\end{align}
Here the vector $\dot {\rm m}\left(\vartheta ,s\right)=\left(\dot
m_1\left(\vartheta ,t\right),\dot m_2\left(\vartheta ,t\right)\right)^\top
$. Remark that the vector $\dot {\rm m}\left(\vartheta ,s\right) $ can be
replaced by the vector $\dot {\rm y}\left(\vartheta ,s\right)=\left(\dot {
  y}_1\left(\vartheta ,s\right),\dot { y}_2\left(\vartheta ,s\right)
\right)^\top $ and the asymptotic properties of the estimator will be the
same.

Define the Gaussian process $\eta _{t }\left(\vartheta _0\right)=\left(\eta
_{1,t }\left(\vartheta _0\right),\eta _{2,t }\left(\vartheta
_0\right)\right)^\top  $
\begin{align*}
\eta _{t }\left(\vartheta _0\right)={\bf I}_\tau ^t\left(\vartheta
\right)^{-1}\int_{\tau }^{t}\frac{f\left(s\right)\dot {\rm y}_s\left(\vartheta
  _0\right)}{\sigma \left(s\right)} \;{\rm d}w\left(s\right), \qquad \tau \leq t\leq T,
\end{align*}
where $\dot {\rm y}_t\left(\vartheta _0\right)=\left(\dot y_1\left(\vartheta
_0,t\right),\dot y_2\left(\vartheta
_0,t \right)\right)^\top$ and   $w\left(s\right),0\leq s\leq T$ is some Wiener process.

\begin{proposition}
\label{P21-z1} Suppose that the conditions ${\cal A}_*$ and ${\cal B}_1-{\cal
  B}_3 $ are satisfied. Then the One-step MLE-process $\vartheta
_{t,\varepsilon }^\star,\tau <t\leq T $ is consistent, asymptotically
normal
\begin{align}
\label{21-z16}
\eta _{t,\varepsilon }\left(\vartheta _0\right)=\left(\eta _{1,t,\varepsilon }\left(\vartheta _0\right),\eta _{2,t,\varepsilon }\left(\vartheta _0\right) \right)^\top\equiv \varepsilon
^{-1}\left(\vartheta _{t,\varepsilon }^\star-\vartheta 
_0\right)\Longrightarrow  {\cal N}\left(0,{\bf
  I}_\tau ^t\left(\vartheta _{0} \right)^{-1}\right)
\end{align}
and the process $\eta _{t,\varepsilon }\left(\vartheta _0\right),\tau
_*\leq t\leq T  $ for any  $\tau _*\in (\tau ,T]$ converges in distribution in $\left({\cal
  C}\left(0,T\right),{\goth B}\right)$ to the Gaussian process
\begin{align}
\label{21-z17}
\eta _{t ,\varepsilon }\left(\vartheta _0\right)\Longrightarrow \eta
_{t  }\left(\vartheta _0\right) ,\qquad \tau _*\leq t\leq T.
\end{align}

\end{proposition}
\begin{proof} 
Consider the difference
\begin{align}
\label{21-z18}
\varepsilon ^{-1}\left(\vartheta _{t,\varepsilon }^\star-\vartheta
_0\right)&=\varepsilon ^{-1}\left(\bar\vartheta _{\tau ,\varepsilon
}-\vartheta _0\right) +{\bf
  I}_\tau ^t\left(\bar\vartheta _{\tau ,\varepsilon } \right)^{-1} \int_{\tau
}^{t}\frac{f\left(s\right)\dot {\rm m}\left(\bar\vartheta _{\tau ,\varepsilon
  },s\right)}{\sigma \left(s\right) }{\rm d}\bar W_s\nonumber\\
&\quad + {\bf
  I}_\tau ^t\left(\bar\vartheta _{\tau ,\varepsilon } \right)^{-1}\int_{\tau
}^{t}\frac{f\left(s\right)^2\dot {\rm m}\left(\bar\vartheta _{\tau ,\varepsilon
  },s\right)}{\varepsilon \sigma \left(s\right)^2 } \left[m\left(\vartheta _{0},s
  \right)-m\left(\bar\vartheta _{\tau ,\varepsilon },s
  \right)\right]{\rm d}s.
\end{align}
The consistency of the estimator $\bar\vartheta _{\tau ,\varepsilon } $ and
the central limit theorem for stochastic integrals allow us to write
\begin{align*}
\pi _{t,\varepsilon }\equiv {\bf I}_\tau ^t\left(\bar\vartheta _{\tau ,\varepsilon } \right)^{-1}
\int_{\tau }^{t}\frac{f\left(s\right)\dot {\rm m}\left(\bar\vartheta _{\tau
    ,\varepsilon },s\right)}{\sigma \left(s\right) }{\rm d}\bar
W_s\Longrightarrow \eta _t\left(\vartheta _0\right)\sim {\cal N}\left(0,{\bf I}_\tau ^t\left(\vartheta _{0}
\right)^{-1} \right).
\end{align*}
If we consider the vector $\pi _{t_1,\varepsilon },\ldots,\pi
_{t_k,\varepsilon } $, where $\tau \leq t_1\leq t_2\leq \ldots\leq t_k\leq T$
then we obtain the convergence
\begin{align*}
\bigl(\pi _{t_1,\varepsilon },\ldots,\pi
_{t_k,\varepsilon }\bigr)\Longrightarrow \bigl(\eta _{t_1}\left(\vartheta
_0\right),\ldots,\eta _{t_k}\left(\vartheta _0\right)\bigr).
\end{align*}

For the ordinary integral we have the relations
\begin{align*}
&{\bf   I}_\tau ^t\left(\bar\vartheta _{\tau ,\varepsilon } \right)^{-1}\int_{\tau
}^{t}\frac{f\left(s\right)^2\dot {\rm m}\left(\bar\vartheta _{\tau ,\varepsilon
  },s\right)}{\varepsilon \sigma \left(s\right)^2 } \left[m\left(\vartheta _{0},s
  \right)-m\left(\bar\vartheta _{\tau ,\varepsilon },s
  \right)\right]{\rm d}s\\
&\qquad \qquad ={\bf   I}_\tau ^t\left(\vartheta _{0} \right)^{-1}\int_{\tau
}^{t}\frac{f\left(s\right)^2\dot {\rm m}\left(\vartheta _{0
  },s\right)}{\varepsilon \sigma \left(s\right)^2 } \left[m\left(\vartheta _{0},s
  \right)-m\left(\bar\vartheta _{\tau ,\varepsilon },s
  \right)\right]{\rm d}s\left(1+O\left(\varepsilon ^{2/3}\right)\right)\\
&\qquad \qquad ={\bf   I}_\tau ^t\left(\vartheta _{0} \right)^{-1}\int_{\tau
}^{t}\frac{f\left(s\right)^2\dot {\rm m}\left(\vartheta _{0
  },s\right)}{\sigma \left(s\right)^2 } \dot {\rm m}\left(\vartheta _{0
  },s\right)^\top{\rm d}s\frac{\left(\vartheta _0-\bar\vartheta _{\tau ,\varepsilon
  }\right)}{\varepsilon } \left(1+O\left(\varepsilon ^{2/3}\right)\right)\\
&\qquad \qquad =\frac{\left(\vartheta _0-\bar\vartheta _{\tau ,\varepsilon
  }\right)}{\varepsilon }+ \frac{\left(\vartheta _0-\bar\vartheta _{\tau ,\varepsilon
  }\right)}{\varepsilon } O\left(\varepsilon ^{2/3}\right)
  =\frac{\left(\vartheta _0-\bar\vartheta _{\tau ,\varepsilon 
  }\right)}{\varepsilon }+O\left(\varepsilon ^{1/3}\right)
\end{align*}
because
\begin{align*}
\int_{\tau
}^{t}\frac{f\left(s\right)^2\dot {\rm m}\left(\vartheta _{0
  },s\right)\dot {\rm m}\left(\vartheta _{0
  },s\right)^\top}{\sigma \left(s\right)^2 } {\rm d}s\longrightarrow {\bf
  I}_\tau ^t\left(\vartheta _{0} \right). 
\end{align*}

The substitution of these limits in \eqref{21-z18}  proves  \eqref{21-z16}.

To prove the weak convergence \eqref{21-z17} we have to verify the estimate
\begin{align*}
\Ex_{\vartheta _0} \left\|\eta _{t_2,\varepsilon }\left(\vartheta
_0\right)-\eta _{t_1,\varepsilon }\left(\vartheta _0\right)\right\| ^4\leq C\,\left|t_2-t_1\right|^2.
\end{align*}
The calculations are direct but cumbersome. It uses the non deneracy of the
Fisher information,  boundedness of the
derivatives and the elementary properties of stochastic and ordinary integrals.
\qed
\end{proof}

Remark that the equations of the adaptive filter \eqref{21-z19},\eqref{21-z20}
are given in recurrent form. It is possible to write the One-step MLE-process in
the recurrent form too. To simplify its calculation we replace $\dot {\rm
  m}\left(\bar\vartheta_{\tau ,\varepsilon },s \right)$ by $\dot {\rm
  y}\left(\bar\vartheta_{\tau ,\varepsilon },s \right)$ in
\eqref{21-z15}. Then we  multiply  both sides of obtained equation  by
${\bf I}_\tau 
^t\left(\bar\vartheta_{\tau ,\varepsilon } \right) $, denote the matrix
\begin{align*}
{\bf L}\left(\vartheta ,t\right)=\frac{f\left(t\right)^2\dot {\rm
  y}\left(\vartheta,t \right)\dot {\rm
  y}\left(\vartheta,t \right)^\top }{\sigma \left(t\right)^2}
\end{align*}
and take the differential of the equation 
\begin{align*}
{\bf I}_\tau
^t\left(\bar\vartheta_{\tau ,\varepsilon } \right)\vartheta _{t,\varepsilon }^\star={\bf I}_\tau
^t\left(\bar\vartheta_{\tau ,\varepsilon } \right)\bar\vartheta _{\tau ,\varepsilon }+ \int_{\tau
}^{t}\frac{f\left(s\right)\dot {\rm y}\left(\bar\vartheta _{\tau ,\varepsilon
  },s\right)}{\sigma \left(s\right)^2 }\left[{\rm
    d}X_s-f\left(s\right)m\left(\bar\vartheta _{\tau ,\varepsilon },s
  \right){\rm d}s\right].
\end{align*}
Then we obtain
\begin{align*}
{\bf I}_\tau ^t\left(\bar\vartheta_{\tau ,\varepsilon } \right) {\rm
  d}\vartheta ^\star_{t,\varepsilon }+{\bf L}\left(\bar\vartheta_{\tau
  ,\varepsilon } ,t\right)\vartheta ^\star_{t,\varepsilon }{\rm d}t&={\bf
  L}\left(\bar\vartheta_{\tau ,\varepsilon } ,t\right)\bar\vartheta_{\tau
  ,\varepsilon }\\
&\qquad +\frac{f\left(s\right)\dot {\rm y}\left(\bar\vartheta _{\tau ,\varepsilon
  },s\right)}{\sigma \left(s\right)^2 }\left[{\rm
    d}X_s-f\left(s\right)m\left(\bar\vartheta _{\tau ,\varepsilon },s
  \right){\rm d}s\right]
\end{align*}
and
\begin{align*}
 {\rm
  d}\vartheta ^\star_{t,\varepsilon }&={\bf I}_\tau ^t\left(\bar\vartheta_{\tau
   ,\varepsilon } \right)^{-1}{\bf L}\left(\bar\vartheta_{\tau 
  ,\varepsilon } ,t\right)\left[\bar\vartheta_{\tau
  ,\varepsilon }-\vartheta ^\star_{t,\varepsilon }\right]{\rm d}t\\
&\qquad +{\bf I}_\tau ^t\left(\bar\vartheta_{\tau
   ,\varepsilon } \right)^{-1}\frac{f\left(t\right)\dot {\rm y}\left(\bar\vartheta _{\tau ,\varepsilon
  },t\right)}{\sigma \left(t\right)^2 }\left[{\rm
    d}X_t-f\left(t\right)m\left(\bar\vartheta _{\tau ,\varepsilon },t  \right){\rm d}s\right].
\end{align*}
The Fisher information can be written in the recurrent form.

\subsubsection{Adaptive filter}

Suppose that we already have the estimators $\bar\vartheta _{\tau ,\varepsilon
} $ and $\vartheta _{t,\varepsilon }^\star=\left(\theta _{1,t,\varepsilon
}^\star,\theta _{2,t,\varepsilon }^\star \right)^\top,\tau <t\leq T $. Then the
adaptive filter  $m^\star_{t,\varepsilon },\tau_* <t\leq T $, where $\tau _*>\tau $ is
\begin{align}
\label{21-z19}
{\rm d}m^\star_{t,\varepsilon }&=a\left(\theta _{2,t,\varepsilon
}^\star,t\right)m^\star_{t,\varepsilon }{\rm d}t+\frac{f\left(t\right)\gamma
  _*^\star\left(t \right)}{\sigma \left(t\right)^2}\left[{\rm
    d}X_t-f\left(t\right)m^\star_{t,\varepsilon }{\rm d}t\right],\quad
m^\star_{\tau_* ,\varepsilon },\\ 
\frac{{\rm d}\gamma _*^\star\left(t \right)
}{{\rm d}t}&=2a\left(\theta _{2,t,\varepsilon }^\star,t\right)\gamma
_*^\star\left(t \right)-\frac{f\left(t\right)^2\gamma 
  _*^\star\left(t \right)^2}{\sigma \left(t\right)^2}+b\left(t\right)^2,\qquad
\gamma _*^\star\left(\tau_*  \right). 
\label{21-z20}
\end{align}
The  initial value  $\gamma _*^\star\left(\tau_*  \right) $  is the solution of
the equation
\begin{align}
\label{21-z21}
\frac{{\rm d}\gamma _*^\star\left(s \right) }{{\rm d}s}&=2a\left(\theta
_{2,\tau ,\varepsilon }^\star,s\right)\gamma _*^\star\left(s
\right)-\frac{f\left(s\right)^2\gamma _*^\star\left(s \right)^2}{\sigma
  \left(s\right)^2}+b\left(s\right)^2,\qquad \gamma _*^\star\left(0 \right)=0,\quad 0\leq s\leq \tau _*
\end{align}
at the point $s=\tau_* $. 

To calculate $m^\star_{\tau_* ,\varepsilon } $ we
denote $$R\left(\theta _2,t\right)=\Phi \left(\theta _2,0,t\right)^{-1}
f\left(t\right)\gamma _*\left(\theta _2,t \right)\sigma
\left(t\right)^{-2}=\Phi \left(\theta _2,0,t\right)^{-1}B\left(\theta _2,t \right) $$ and
remark that
\begin{align*}
m\left(\vartheta ,t\right)&=\Phi \left(\theta _2,0,t\right)\left[\theta
  _1+\int_{0}^{t}R\left(\theta _2,s\right){\rm d}X_s\right]\\
&= \Phi \left(\theta _2,0,t\right)\left[\theta
  _1+R\left(\theta _2,t\right)X_t -    \int_{0}^{t}R'_s\left(\theta _2,s\right)X_s{\rm d}s\right]\\
&=\theta
  _1\Phi \left(\theta _2,0,t\right)+\frac{f\left(t\right)\gamma _*\left(\theta _2,t
\right)}{\sigma \left(t\right)^{2}}X_t -    \int_{0}^{t}\Phi \left(\theta
_2,s,t\right)\left[ A'_s\left(\theta _2,s\right)+  B'_s\left(\theta
_2,s\right)\right]X_s{\rm d}s .
\end{align*}
Here $R'_s\left(\vartheta ,s\right) $ and  $A'_s\left(\vartheta ,s\right) $
are the derivatives of $R\left(\vartheta ,s\right) ,A\left(\vartheta ,s\right)
$ w.r.t. $s$.

The initial value is calculated as follows
\begin{align}
\label{21-z22}
m^\star_{\tau_* ,\varepsilon }&=\theta^\star_{1,\tau_*,\varepsilon } \Phi
\left(\theta^\star_{2,\tau_*,\varepsilon
},0,\tau_*\right)+\frac{f\left(\tau_*\right)\gamma
  _*\left(\theta^\star_{2,\tau_*,\varepsilon },\tau_* \right)}{\sigma
  \left(\tau_*\right)^{2}}X_{\tau_*}\nonumber\\
&\quad  - \int_{0}^{\tau_*}\Phi \left(\theta
^\star_{2,\tau_*,\varepsilon },s,\tau_*\right)\left[
  A'_s\left(\theta^\star_{2,\tau_*,\varepsilon },s\right)+ B'_s\left(\theta
  ^\star_{2,\tau_*,\varepsilon },s\right)\right]X_s{\rm d}s.
\end{align}
\begin{theorem}
\label{T21-z1} Let the conditions ${\cal A}_*$ and ${\cal B}_1-{\cal B}_3$
hold. Then the error of approximation
\begin{align*}
\varepsilon ^{-1}\left(m^\star_{t ,\varepsilon
}-m\left(\vartheta_0,t\right)\right) \Longrightarrow \xi _m\left(\vartheta
_0,t\right),\qquad \tau _*\leq t\leq T 
 \end{align*}
where the Gaussian process $\xi _m\left(\vartheta
_0,t\right), \tau _*\leq t\leq T  $  is defined below in \eqref{21-z23}.

\end{theorem}

\begin{proof}
Denote $\delta _m\left(t\right)=m^\star_{t ,\varepsilon }-m\left(\vartheta
_0,t\right)$, $\delta _\vartheta \left(t\right)=\vartheta _{t,\varepsilon
}^\star-\vartheta _0$, $\delta _{\theta_2} \left(t\right)=\theta _{2,t,\varepsilon
}^\star-\theta _{2,0}$, $\delta _\gamma \left(t\right)=\gamma
_*^\star\left(t\right)-\gamma _*\left(\vartheta _{2,0},t\right)$. We have
already the estimate $\delta _\vartheta \left(t\right)=O\left(\varepsilon
\right)$. The equations for $\delta _m\left(t\right) $ and $\delta _\gamma
\left(t\right)$ we obtain as difference of  equations for corresponding
functions
\begin{align*}
{\rm d}\delta _m\left(t\right)&=A\left(\theta^\star _{2,t,\varepsilon},t\right)\delta _m\left(t\right){\rm d}t +\left[a\left(\theta^\star _{2,t,\varepsilon},t\right)-a\left(\theta _{2,0},t\right)\right]m\left(\vartheta
_0\right){\rm d}t \\
&\qquad +\varepsilon \,\frac{f\left(t\right)\delta _\gamma
  \left(t\right)}{\sigma \left(t\right)}\;{\rm d}\bar W_t ,\qquad \quad \delta
_m\left(\tau _*\right) ,\\
\frac{{\rm d}\delta _\gamma \left(t\right)}{{\rm
    d}t}&=2\left[a\left(\theta^\star _{2,t,\varepsilon
},t\right)-a\left(\theta _{2,0  },t\right)\right]  \gamma_* \left(\theta
_{2,0},t\right)\\
&\qquad + \left[2a\left(\theta _{2,0
  },t\right)-\frac{f\left(t\right)^2 \left[\gamma
_*^\star\left(t\right)+\gamma _*\left(\vartheta _{2,0},t\right)\right]
  }{\sigma \left(t\right)^2}\right] \delta _\gamma \left(t\right),\quad \delta
_\gamma \left(\tau _*\right). 
\end{align*}
We have $a\left(\theta^\star _{2,t,\varepsilon },t\right)-a\left(\theta _{2,0
},t\right)=\dot a\left(\theta _{2,0 },t\right)\delta _{\theta _{2 }}
\left(t\right)\left(1+O\left(\varepsilon \right)\right) $. Using these
relations, boundedness of the derivatives and Gr\"onwell lemma it can be shown
that $\delta _\gamma \left(t\right)=O\left(\varepsilon \right), \delta
_m\left(t\right)=O\left(\varepsilon \right)$. 

Therefore
\begin{align*}
{\rm d}\delta _m\left(t\right)&=A\left(\theta _{2,0},t\right)\delta
_m\left(t\right){\rm d}t +\dot a\left(\theta _{2,0},t\right)m\left(\vartheta
_0,t\right)\delta _\vartheta \left(t\right){\rm d}t  +\varepsilon \,\frac{f\left(t\right)\delta _\gamma
  \left(t\right)}{\sigma \left(t\right)}\;{\rm d}\bar W_t+O\left(\varepsilon
^2\right) ,\\
\frac{{\rm d}\delta _\gamma \left(t\right)}{{\rm
    d}t}&=2\dot a\left(\theta _{2,0  },t\right)  \gamma_* \left(\theta
_{2,0},t\right)\delta _\vartheta \left(t\right) + 2A\left(\theta _{2,0
  },t\right)\delta _\gamma \left(t\right)+O\left(\varepsilon ^2\right). 
\end{align*}
The solutions of these linear equations are
\begin{align*}
\delta _\gamma \left(t\right)&=\delta _\gamma \left(\tau _*\right)\Phi \left(\theta
_{2,0},\tau _*,t\right)^2+2\int_{\tau _*}^{t} \Phi \left(\theta
_{2,0},s,t\right)^2\dot a\left(\theta _{2,0  },s\right)  \gamma_* \left(\theta
_{2,0},s\right)\delta _\vartheta \left(s\right){\rm d}s+O\left(\varepsilon
^2\right),\\
\delta _m\left(t\right)&=\delta _m\left(\tau _*\right)\Phi \left(\theta
_{2,0},\tau _*,t\right)+\int_{\tau _*}^{t} \Phi \left(\theta
_{2,0},s,t\right)\dot a\left(\theta _{2,0},s\right)m\left(\vartheta
_0,s\right)\delta _\vartheta \left(s\right){\rm d}s\\
&\qquad +\varepsilon \int_{\tau _*}^{t} \Phi \left(\theta
_{2,0},s,t\right){f\left(s\right)\delta _\gamma
  \left(s\right)}{\sigma \left(s\right)^{-1}}\;{\rm d}\bar W_s+O\left(\varepsilon
^2\right).
\end{align*}
Solution of the equation related with \eqref{21-z21} gives us  the representation
\begin{align*}
\varepsilon^{-1}\delta _\gamma \left(\tau
_*\right)&=2\varepsilon^{-1}\int_{0}^{\tau _*}\Phi \left(\theta 
_{2,0},s,\tau _*\right)^2\dot a\left(\theta _{2,0},s\right)\gamma _*\left(\theta
_{2,0},s\right)\delta _{\theta _2}\left(\tau _*\right){\rm d}s +O\left(\varepsilon
\right)\\ 
&\Longrightarrow  2\int_{0}^{\tau _*}\Phi \left(\theta
_{2,0},s,\tau _*\right)^2\dot a\left(\theta _{2,0},s\right)\gamma _*\left(\theta
_{2,0},s\right)\eta  _{2,\tau _*}\left(\theta _2\right){\rm d}s\equiv \xi
_\gamma \left(\vartheta _0\right). 
\end{align*}   
Hence
\begin{align*}
\varepsilon^{-1}\delta _\gamma \left(t\right)&\Longrightarrow \xi _\gamma
\left(\vartheta _0\right)\Phi \left(\theta
_{2,0},\tau _*,t\right)^2+2\int_{\tau _*}^{t} \Phi \left(\theta
_{2,0},s,t\right)^2\dot a\left(\theta _{2,0  },s\right)  \gamma_* \left(\theta
_{2,0},s\right)\eta _{2,s} \left(\vartheta _0\right){\rm d}s\\
&\equiv \xi _\gamma \left(\vartheta _0,t\right),\qquad \qquad \tau _*\leq t\leq T.
\end{align*}

Let us denote $C\left(\theta _2,s\right)=A'_s\left(\theta_{2},s\right)+ B'_s\left(\theta
  _{2 },s\right) $ then we can write (see \eqref{21-z22}) 
\begin{align*}
\varepsilon ^{-1}\delta _m\left(\tau _*\right)&\Longrightarrow \Phi
\left(\theta _{2,0},0,\tau _*\right)\eta _{1,\tau _*}\left(\vartheta
_0\right)+\theta _{1,0}\dot \Phi \left(\theta _{2,0},0,\tau _*\right)\eta
_{2,\tau _*}\left(\vartheta _0\right)+\frac{f\left(\tau _*\right)}{\sigma \left(\tau
  _*\right)^2}\xi _\gamma \left(\vartheta _0\right) x_{\tau _*}\left(\vartheta _0\right)\\ &\quad
-\eta _{2,\tau _*}\left(\vartheta
_0\right)\,\int_{0}^{\tau _*}\left[\dot \Phi \left(\theta _{2,0},0,\tau
  _*\right)C\left(\theta _{2,0},s\right)+\Phi \left(\theta _{2,0},0,\tau
  _*\right)\dot C\left(\theta _{2,0},s\right)\right] x_s\left(\vartheta _0\right)\,{\rm d}s \\
&\equiv \xi _m\left(\vartheta _0\right). 
\end{align*}

Therefore
\begin{align}
\label{21-z23}
\varepsilon ^{-1}\delta _m\left(t\right)&\Longrightarrow   \xi _m\left(\vartheta _0\right)\Phi \left(\theta
_{2,0},\tau _*,t\right)+\int_{\tau _*}^{t} \Phi \left(\theta
_{2,0},s,t\right)\dot a\left(\theta _{2,0},s\right)y_s\left(\vartheta
_0\right)\eta  _{2,s} \left(\vartheta _0\right){\rm d}s\nonumber\\
&\quad + \int_{\tau _*}^{t} \Phi \left(\theta
_{2,0},s,t\right){f\left(s\right)}{\sigma \left(s\right)^{-1}}\xi _\gamma
\left(\vartheta _0,s\right)\;{\rm d}w\left(s\right) \equiv \xi _m\left(\vartheta _0,t\right),\quad \tau _*\leq t\leq T.
\end{align}
\end{proof}

\subsection{Random initial value}

Let us consider an example of partially observed system with the random
initial value of non observed component:
\begin{align}
\label{21-71}
{\rm d}X_t&=fY_t\,{\rm d}t+\varepsilon \sigma \,{\rm d}W_t,\qquad X_0=0,\qquad
0\leq t\leq T,\\
\label{21-72}
{\rm d}Y_t&=\vartheta Y_t\,{\rm d}t+\varepsilon b\,{\rm d}V_t,\qquad \quad Y_0=y_0\sim
{\cal N}\left(0,d^2\right).
\end{align}

Here $f>0,\sigma >0,b>0,d^2>0$ and $\vartheta \in\Theta =\left(\alpha ,\beta
\right), \alpha >0$.  The random variable $Y_0$ in independent on   the Wiener
processes  $W_t,V_t;0\leq t\leq T$.  The sample value   of $Y_0$ we denote
$y_0$. As usual, we have to estimate the parameter $\vartheta $
by observations $X^T=\left(X_t,0\leq t\leq T\right)$.

\subsubsection{Parameter estimators}
Remark that these processes can be written in explicit form 
\begin{align*}
X_t&=\frac{fy_0}{\vartheta }\left[e^{\vartheta t}-1\right]+\varepsilon
fb\int_{0}^{t}\int_{0}^{s}e^{\vartheta\left( s-r\right)} 
{\rm d}V_r\,{\rm d}s +\varepsilon \sigma W_t, \qquad
0\leq t\leq T, \\
Y_t&=y_0\,e^{\vartheta t}+\varepsilon b\int_{0}^{t}e^{\vartheta \left(t-s\right)}{\rm d}V_s.
\end{align*}

The limit system  ($\vartheta _0$ is the true value)
\begin{align*}
\frac{{\rm d}x_t\left(\vartheta _0\right)}{{\rm d}t}&=f\,y_t\left(\vartheta
_0\right),\qquad x_0\left(\vartheta _0\right)=0,\qquad 0\leq t\leq T,\\
\frac{{\rm d}y_t\left(\vartheta _0\right)}{{\rm d}t}&=\vartheta_0 \,y_t\left(\vartheta
_0\right),\qquad y_0\left(\vartheta _0\right)\sim{\cal N}\left(0,d^2\right)
\end{align*}
gives us a couple of  random processes $x_t\left(\vartheta
_0\right),y_t\left(\vartheta _0\right),0\leq t\leq T$. As $y_t\left(\vartheta
\right)=y_0e^{\vartheta t}$, we have the equality
\begin{align*}
x_t\left(\vartheta \right)=\frac{y_0 f}{\vartheta }\left(e^{\vartheta
  t}-1\right), \qquad 0\leq t\leq T.
\end{align*}
There are several ways to construct a preliminary estimator of $\vartheta
_0$. Note that we can not   define the MDE by the equality
\begin{align*}
\check\vartheta _{\tau,\varepsilon  }=\arg\inf_{\vartheta \in\Theta
}\int_{0}^{\tau }\left[X_t-x_t\left(\vartheta \right)\right]^2{\rm d}t 
\end{align*}
because $y_0$ is unknown and it's conditional expectation is always
$m\left(\vartheta ,0\right)=0$.

Let us denote $x'_t\left(\vartheta \right)$ the derivative of
$x_t\left(\vartheta \right)$  w.r.t.  $t$ and  remark that  ($t_1<t_2$)
\begin{align*}
\vartheta_0 =\frac{\ln \left|x'_{t_2}\left(\vartheta_0 \right)\right|-\ln
\left|x'_{t_1}\left(\vartheta_0 \right)\right|}{ \left|t_2-t_1\right| }. 
\end{align*}
Therefore if we have a consistent estimator $X'_{t,\varepsilon }$ of the
derivative, then the MME
\begin{align}
\label{MME}
\vartheta _{t_2,\varepsilon}^*= \frac{\ln \left|X'_{t_2,\varepsilon }\right|-\ln
\left|X'_{t_1,\varepsilon }\right|}{ \left|t_2-t_1\right| }
\end{align}
can be  consistent too. 

Note that  for the estimator of the derivative $X'_{t,\varepsilon
}=\left({X_{t+\delta _\varepsilon }-X_t}\right)/{\delta _\varepsilon } $ we have
\begin{align*}
X'_{t,\varepsilon }&=\frac{fy_0}{\delta _\varepsilon }\int_{t}^{t+\delta
  _\varepsilon }e^{\vartheta_0 s}{\rm d}s+\frac{fb\varepsilon }{\delta _\varepsilon
}\int_{t}^{t+\delta _\varepsilon }\int_{0}^{s}e^{\vartheta_0\left(s-r\right)}{\rm
  d}V_r\,{\rm d}s+\frac{\varepsilon }{\delta _\varepsilon }\sigma \left[W_{t+\delta
    _\varepsilon }-W_t\right]\\
&=fy_0 e^{\vartheta_0t}+\frac{f\vartheta_0  y_0}{2}e^{\vartheta_0t}\delta
_\varepsilon +fb\varepsilon  \int_{0}^{t}e^{\vartheta_0\left(t-r\right)}{\rm
  d}V_r+\frac{\varepsilon }{\sqrt{\delta  _\varepsilon} }\sigma\; \xi _{t,\varepsilon}
+O\left(\varepsilon \sqrt{\delta  _\varepsilon} \right)
\end{align*} 
where $ \xi _{t,\varepsilon}\sim {\cal N}\left(0,1\right) $. The correct  rate
$\delta _\varepsilon=\varepsilon ^{2/3} $ we obtain as solution of the equation $ \delta _\varepsilon =\varepsilon /\sqrt{\delta _\varepsilon }$.
Taylor expansion allows to write
\begin{align*}
 \ln\left|X'_{t,\varepsilon }\right|&=\ln \left| fy_0
 e^{\vartheta_0t}\left(1+\left[\frac{\vartheta _0}{2}+\frac{\sigma\;\xi
     _{t,\varepsilon}}{fy_0}e^{-\vartheta_0t} \right] \varepsilon
 ^{2/3}\left(1+o\left(1\right)\right)  \right)\right| \\
 &=\ln \left(f\left|y_0\right|\right)+\vartheta _0t+
\ln \left|1+\left[\frac{\vartheta _0}{2}+\frac{\sigma\;\xi
     _{t,\varepsilon}}{fy_0}e^{-\vartheta_0t} \right] \varepsilon
 ^{2/3}\left(1+o\left(1\right)\right)\right|\\
 &=\ln \left(f\left|y_0\right|\right)+\vartheta _0t+
\left[\frac{\vartheta _0}{2}+\frac{\sigma\;\xi
     _{t,\varepsilon}}{fy_0}e^{-\vartheta_0t} \right] \varepsilon
 ^{2/3}\left(1+o\left(1\right)\right),
\end{align*}
where the last equality is valid asymptotically as $\varepsilon \rightarrow
0.$

Therefore we proved the following proposition. 

\begin{proposition}
\label{P21-z2} 
For the MME $\vartheta _{t_2,\varepsilon}^* $     the following
relation holds
\begin{align*}
\varepsilon ^{-2/3} \left(\vartheta _{t_2,\varepsilon}^*-\vartheta _0\right)
&\Longrightarrow
\frac{\sigma}{f\left(t_2-t_1\right)}\left[\frac{\xi _{t_2}}{y_0}e^{-\vartheta_0t_2}
  -\frac{\xi _{t_1}}{y_0}e^{-\vartheta_0t_1} \right].
\end{align*}
  Here $\xi _{t_1}\sim {\cal N}\left(0,1\right) ,\xi _{t_2}\sim {\cal
  N}\left(0,1\right) $ and $y_0\sim {\cal N}\left(0,d^2\right)$ are
independent Gaussian random variables. 
\end{proposition}
Remark that $\xi _{t_1}d/y_0$ and $\xi _{t_2}d/y_0$ are Cauchy variables and
hence there is no convergence of moments of this estimator.

It is easy to see that $ \varepsilon ^{-2/3}y_0 \left(\vartheta
_{t_2,\varepsilon}^*-\vartheta _0\right)$ is asymptotically normal.

\begin{remark}
\label{R21-11}
{\rm 
Introduce $t_{i,\varepsilon }=p_i/\ln \left(\frac{1}{\varepsilon
}\right)\rightarrow 0, i=1,2$. Here $p_1>0, p_2>p_1$ and $0<t_{i,\varepsilon
}<T $.   Then the normalizing and limit change   
\begin{align}
\label{21-73t}
\ln\left(\frac{1}{\varepsilon }\right)\varepsilon ^{-2/3} \left(\vartheta _{t_2,\varepsilon}^*-\vartheta _0\right)
&\Longrightarrow
\frac{\sigma}{f\left(p_2-p_1\right)}\left[\frac{\xi _{p_2}}{y_0}
  -\frac{\xi _{p_1}}{y_0} \right].
\end{align}
Here $\xi _{p_1}\sim {\cal N}\left(0,1\right)$, $\xi _{p_2}\sim {\cal
  N}\left(0,1\right)$   are idependent random ariables. 

}

\end{remark}

The Kalman-Bucy filter for conditional expectation $m\left(\vartheta
,t\right),0\leq t\leq T$  is 
\begin{align}
\label{21-73}
{\rm d}m\left(\vartheta ,t\right)&=\vartheta  m\left(\vartheta ,t\right){\rm
  d}t+\frac{f\gamma \left(\vartheta ,t\right)}{\varepsilon ^2\sigma ^2}\left[{\rm
    d}X_t-fm\left(\vartheta ,t\right){\rm d}t\right] ,\quad m\left(\vartheta
,0\right)=0,\\
\frac{\partial \gamma \left(\vartheta ,t\right)}{\partial t}&=2\vartheta
\gamma \left(\vartheta ,t\right) -\frac{f^2\gamma \left(\vartheta
  ,t\right)^2}{\varepsilon ^2\sigma ^2}+\varepsilon ^2b^2,\qquad\qquad \gamma \left(\vartheta
  ,0\right)=d^2. \label{21-74}
\end{align}

The solution of the equation \eqref{21-73} can be written explicitly
\begin{align*}
m\left(\vartheta ,t\right)&=\int_{0}^{t}\Phi \left(\vartheta
,s,t\right)\frac{f^2\gamma \left(\vartheta ,s\right)}{\varepsilon ^2\sigma ^2}
\,Y_s\,{\rm d}s+\int_{0}^{t}\Phi \left(\vartheta ,s,t\right)\frac{f\gamma
  \left(\vartheta ,s\right)}{\varepsilon \sigma } \,{\rm d}W_s\\
&=y_0\int_{0}^{t}\Phi \left(\vartheta
,s,t\right)\frac{f^2\gamma \left(\vartheta ,s\right)}{\varepsilon ^2\sigma ^2}
\,e^{\vartheta _0s}\,{\rm d}s+\int_{0}^{t}\Phi \left(\vartheta ,s,t\right)\frac{f\gamma
  \left(\vartheta ,s\right)}{\varepsilon \sigma } \,{\rm d}W_s\\
&\qquad +\int_{0}^{t}\Phi \left(\vartheta
,s,t\right)\frac{f^2\gamma \left(\vartheta ,s\right)}{\varepsilon \sigma ^2}
\,\int_{0}^{s}e^{\vartheta _0\left(s-r\right)} {\rm d}V_r\,{\rm d}s
\end{align*}
where
\begin{align*}
\Phi \left(\vartheta ,s,t\right)=\exp\left(\int_{s}^{t}\left[\vartheta -
  \frac{f^2\gamma \left(\vartheta ,r\right)}{\varepsilon ^2\sigma
    ^2}\right]{\rm d}r \right). 
\end{align*}
Let us see the asymptotic of the main term 
\begin{align*}
I\left(\tau ,\varepsilon \right)\equiv y_0\int_{0}^{\tau }\Phi \left(\vartheta
,s,\tau \right)\frac{f^2\gamma \left(\vartheta ,s\right)}{\varepsilon ^2\sigma ^2}
\,e^{\vartheta _0s}\,{\rm d}s
\end{align*}
for the small values of $\tau $. Integrating by parts we obtain 
\begin{align}
\label{21-in}
I\left(\tau ,\varepsilon \right)&=y_0\Phi \left(\vartheta
,0,\tau \right)\int_{0}^{\tau }\exp\left(\int_{0}^{s}
  \frac{f^2\gamma \left(\vartheta ,r\right)}{\varepsilon ^2\sigma
    ^2}{\rm d}r \right)\frac{f^2\gamma \left(\vartheta
  ,s\right)}{\varepsilon ^2\sigma ^2} 
\,e^{\left(\vartheta _0-\vartheta\right) s}\,{\rm d}s\nonumber\\
&=y_0\Phi \left(\vartheta
,0,\tau \right)\int_{0}^{\tau }
\,e^{\left(\vartheta _0-\vartheta\right) s}\,{\rm d}\exp\left(\int_{0}^{s}
  \frac{f^2\gamma \left(\vartheta ,r\right)}{\varepsilon ^2\sigma
    ^2}{\rm d}r \right)\nonumber\\
&=y_0\,e^{\vartheta _0\tau }-y_0\Phi \left(\vartheta
,0,\tau \right)-y_0\left(\vartheta _0-\vartheta \right)\int_{0}^{\tau }\Phi
  \left(\vartheta 
,s,\tau \right) \,e^{\vartheta _0 s}\,{\rm d}s\nonumber\\
&=y_0\,e^{\vartheta _0\tau }+O\left(\varepsilon ^2\right).
\end{align}

The function   $\gamma _*\left(\vartheta ,t\right)=\varepsilon ^{-2}\gamma
\left(\vartheta ,t\right)$  is solution of the Riccati equation
\begin{align*}
 \frac{\partial \gamma_* \left(\vartheta ,t\right)}{\partial t}&=2\vartheta
\gamma_* \left(\vartheta_* ,t\right) -\frac{f^2\gamma_* \left(\vartheta
  ,t\right)^2}{\sigma ^2}+b^2,\qquad\qquad \gamma_* \left(\vartheta
  ,0\right)=\frac{d^2}{\varepsilon ^2 }. 
\end{align*}
Recall that Riccati  equation with constant coefficients has the following  explicite
solution (see  \cite{A83})
\begin{align*}
\gamma_* \left(\vartheta ,t\right)=e^{-2r\left(\vartheta
  \right)t}\left[\frac{1}{\frac{d^2}{\varepsilon ^2}-\frac{\sigma ^2}{f^2}
    \left(\vartheta +r\left(\vartheta \right)\right)}+\frac{f^2\left(1-e^{-2r\left(\vartheta
  \right)t}\right)}{2\sigma
    ^2r\left(\vartheta \right)}    \right]^{-1}+\frac{\sigma ^2}{f^2}
    \left(\vartheta +r\left(\vartheta \right)\right).
\end{align*}
Here $r\left(\vartheta \right)=\left(\vartheta ^2+b^2f^2\sigma ^{-2}\right)^{1/2}$.
The function $\gamma_* \left(\vartheta ,t\right)$ for any  $t\in (0,T]$ has
the asymptotic expansion
\begin{align*}
\gamma_* \left(\vartheta ,t\right)=\frac{\sigma ^2}{f^2}\left[\vartheta
  +r\left(\vartheta \right) +\frac{2r\left(\vartheta \right)}{e^{2r\left(\vartheta
  \right)t} -1 }\right]\left(1+O\left(\varepsilon ^2\right)\right)=\gamma^*
\left(\vartheta ,t\right)\left(1+O\left(\varepsilon ^2\right)\right) 
\end{align*}
where the function $\gamma^* \left(\vartheta ,t\right) $ is defined by the
last equality. Moreover, for any $\tau \in (0,T]$ we have the uniform
  convergence too
\begin{align*}
\sup_{\tau \leq t\leq T}\left|\gamma_* \left(\vartheta ,t\right)-\gamma^*
\left(\vartheta ,t\right)\right| \leq C\varepsilon ^2 .
\end{align*}
Therefore for any $t\in \left(0,T\right)$ we have the approximation
\begin{align*}
\gamma \left(\vartheta ,t\right)=\varepsilon ^2\gamma^*
\left(\vartheta ,t\right)\left(1+O\left(\epsilon ^2\right)\right).
\end{align*}

Omitting the terms of order $O\left(\varepsilon ^2\right)$ we obtain the
representation
\begin{align*}
m^*\left(\vartheta ,t\right)&=m\left(\vartheta ,\tau \right)\Phi^* \left(\vartheta
,\tau ,t\right)+ y_0\int_{\tau}^{t}\Phi^* \left(\vartheta
,s,t\right)\frac{f^2\gamma^* \left(\vartheta ,s\right)}{\sigma ^2}
\,e^{\vartheta _0s}\,{\rm d}s\\
&\qquad+\varepsilon \int_{\tau}^{t}\Phi^* \left(\vartheta ,s,t\right)\frac{f\gamma
  \left(\vartheta ,s\right)}{ \sigma } \,{\rm d}W_s\\
&\qquad +\varepsilon \int_{\tau}^{t}\Phi^* \left(\vartheta
,s,t\right)\frac{f^2\gamma^* \left(\vartheta ,s\right)}{ \sigma ^2}
\,\int_{0}^{s}e^{\vartheta _0\left(s-r\right)} {\rm d}V_r\,{\rm d}s\\
&=y_0e^{\vartheta _0\tau }\Phi^* \left(\vartheta
,\tau ,t\right)+y_0\int_{\tau}^{t}\Phi^* \left(\vartheta
,s,t\right)\frac{f^2\gamma^* \left(\vartheta ,s\right)}{\sigma ^2}
\,e^{\vartheta _0s}\,{\rm d}s \; +o\left(1\right),
\end{align*}
where
\begin{align*}
\Phi^* \left(\vartheta ,s,t\right)=\exp\left(\int_{s}^{t}\left[\vartheta -
  \frac{f^2\gamma^* \left(\vartheta ,r\right)}{\sigma ^2}\right]{\rm d}r
\right). 
\end{align*}
The initial value $m\left(\vartheta ,\tau \right) $ we obtain as solution of
the equation \eqref{21-73}.  

For the derivative $\dot m_t\left(\vartheta \right)=\partial
m_t\left(\vartheta \right)/\partial \vartheta $ we have the equation
\begin{align*}
{\rm d}\dot m_t\left(\vartheta \right)&=\left[\vartheta -\frac{f^2\gamma
    ^*\left(\vartheta ,t\right)}{\sigma ^2}\right]\dot m_t\left(\vartheta
\right) {\rm d}t+\frac{f\dot \gamma
    ^*\left(\vartheta ,t\right)}{\sigma ^2}{\rm d}X_t+\left[1 -\frac{f^2\dot\gamma
    ^*\left(\vartheta ,t\right)}{\sigma ^2}\right] m_t\left(\vartheta
\right) {\rm d}t
\end{align*}
and obtain the expression
\begin{align*}
\dot m_t\left(\vartheta \right)&=y_0\int_{\tau }^{t}\Phi^* \left(\vartheta
,s,t\right)\left[t-s+\frac{f^2}{\sigma ^2}\left(\dot\gamma ^*\left(\vartheta
  ,s\right)-\int_{s}^{t}\dot\gamma ^*\left(\vartheta
  ,s\right){\rm d}r  \right)\right]\;e^{\vartheta _0s}{\rm
  d}s+o\left(1 \right)\\
&=y_0\; H\left(\vartheta ,\vartheta _0,t\right)+o\left(1
\right)\longrightarrow y_0\; H\left(\vartheta ,\vartheta _0,t\right) 
\end{align*} 
as $\varepsilon \rightarrow 0$.

Therefore the empirical Fisher information function  converges in probability
\begin{align*}
{\rm I}_\tau ^t\left(\vartheta _0\right)&=\frac{f^2}{\sigma
  ^2}\int_{\tau }^{t}\dot m_s\left(\vartheta_0 \right)^2 {\rm d}s\longrightarrow \frac{y_0^2\;f^2}{\sigma
  ^2}\int_{\tau }^{t}H\left(\vartheta_0 ,\vartheta _0,s\right) ^2{\rm d}s\nonumber\\
&=y_0^2G\left(\vartheta_0,T \right).
\end{align*}
\begin{remark}
\label{R21-15}
{\rm  The family of measures $\left(\Pb_\vartheta ,\vartheta \in\Theta \right)$
is locally asymptotically mixing normal (LAMN). The normalized likelihood
ratio function constructed by the observations $X_\tau ^T=\left(X_t,\tau \leq t\leq T\right)$ 
$$
Z_\varepsilon \left(\vartheta _0,u\right)=\frac{L\left(\vartheta
_0+\varepsilon u,X_\tau ^T\right)}{  L\left(\vartheta _0,X_\tau^T\right)},\qquad \quad
u\in\UU_\varepsilon =\left(\frac{\alpha -\vartheta _0}{\varepsilon
},\frac{\beta -\vartheta _0}{\varepsilon }\right) 
$$
 admits the representation
\begin{align*}
Z_\varepsilon \left(\vartheta _0,u\right)=\exp\left(u\,y_0\,\Delta
_\varepsilon \left(\vartheta
_0,X_\tau^T\right)-\frac{u^2}{2}y_0^2\;G\left(\vartheta_0,T \right)+r_\varepsilon
\right). 
\end{align*}
Here $r_\varepsilon \rightarrow 0$ and 
\begin{align*}
\Delta _\varepsilon \left(\vartheta _0,X^T\right)\Longrightarrow \Delta
 \left(\vartheta _0\right)\sim {\cal
  N}\left(0,G\left(\vartheta_0,T \right)\right) .
\end{align*}

We do not study the MLE but note that it has the limit Cauchy distribution
\begin{align*}
\frac{\hat\vartheta _\varepsilon -\vartheta _0}{\varepsilon }\Longrightarrow \hat u= \frac{\Delta
 \left(\vartheta _0\right)}{y_0\;G\left(\vartheta_0,T \right)}
\end{align*}
}
\end{remark}

The K-B  filter on the time interval $\left[\tau ,T\right]$ ($\tau >0$) we
re-write
 as follows
\begin{align}
\label{21-75b}
{\rm d}m_*\left(\vartheta ,t\right)&=\left[\vartheta-\frac{f^2\gamma^*
    \left(\vartheta ,t\right)}{\sigma ^2} \right]  m_*\left(\vartheta
,t\right){\rm 
  d}t+\frac{f\gamma^* \left(\vartheta ,t\right)}{\sigma ^2}{\rm
    d}X_t .
\end{align}
The initial value  $m_*\left(\vartheta
,\tau \right)=m\left(\vartheta
,\tau \right) $  can be calculated by solving the system
\eqref{21-73}-\eqref{21-74} on the interval $\left[0,\tau \right]$.

The equation for the derivative $\dot m_*\left(\vartheta
,t\right)$ is 
\begin{align}
\label{21-76}
{\rm d}\dot m_*\left(\vartheta ,t\right)&=\left[1 -{f^2\dot\gamma
    ^*\left(\vartheta ,t\right)}{\sigma 
    ^{-2}}\right]  m_*\left(\vartheta ,t\right){\rm
  d}t+\left[\vartheta -{f^2\gamma ^*\left(\vartheta ,t\right)}{\sigma
    ^{-2}}\right]\dot m_*\left(\vartheta ,t\right){\rm d}t \nonumber\\
&\qquad +{f\dot\gamma ^*\left(\vartheta ,t\right)}{\sigma    ^{-2}}{\rm
    d}X_t ,\quad \qquad \dot
m_*\left(\vartheta ,\tau \right),\quad \tau <t\leq T .
\end{align}

Note that the calcumation of the  $ m_*\left(\vartheta ,t\right) $  and $\dot
m_*\left(\vartheta ,t\right) $ by these equations does not depend on $y_0$,
which appears in the limits only.

If $\varepsilon \rightarrow 0$ then
\begin{align*}
m_*\left(\vartheta ,t\right)\longrightarrow y_t\left(\vartheta ,\vartheta
_0\right)&=y_0e^{\vartheta _0\tau } \left(\vartheta ,\vartheta
_0\right)\Phi^* \left(\vartheta ,\tau ,t\right)\\
&\qquad +y_0\int_{\tau }^{t}\Phi^* \left(\vartheta ,s,t\right) {f^2\gamma
  ^*\left(\vartheta ,s\right)}{\sigma^{-2}}\;e^{\vartheta _0s}\,{\rm
  d}s=y_0\,h_t\left(\vartheta ,\vartheta_0 \right)  ,
\end{align*}
where the function $h_t\left(\vartheta ,\vartheta _0\right)$ is defined by the last
equality.

The expression for the limit of the  derivative  $\dot m\left(\vartheta
,t\right)\longrightarrow \dot y_t\left(\vartheta ,\vartheta
_0\right) $    is the following random  function
\begin{align*}
 \dot y_t\left(\vartheta ,\vartheta
_0\right)&=  \dot y_\tau \left(\vartheta ,\vartheta
_0\right)\Phi^* \left(\vartheta ,0,\tau\right)\\
&\qquad +  y_0\int_{\tau}^{t}\Phi^* \left(\vartheta ,s,t\right)\left[h_s\left(\vartheta ,\vartheta
_0\right)+ \frac{f^2\dot\gamma ^*\left(\vartheta ,s\right)}{\sigma
  ^2}\;\left[e^{\vartheta _0s}-h_s\left(\vartheta ,\vartheta
_0\right)\right]  \right]  {\rm d}s\\
&=\dot y_0e^{\vartheta _0\tau }\tau \Phi^* \left(\vartheta ,0,\tau \right)+y_0\;\bar H_t\left(\vartheta
 ,\vartheta _0\right)\equiv y_0 \,\bar G_t\left(\vartheta ,\vartheta _0\right),
\end{align*}
where $\bar H_t\left(\vartheta ,\vartheta _0\right) $ and $\bar G_t\left(\vartheta
,\vartheta _0\right) $ are defined by the last two
equalities. 
 If $\vartheta =\vartheta _0$, then $h_s\left(\vartheta_0 ,\vartheta
_0\right)=e^{\vartheta _0s} $ and 
\begin{align*}
 \dot y_t\left(\vartheta_0 ,\vartheta
_0\right)&=  \dot y_\tau \left(\vartheta ,\vartheta
_0\right)\Phi^*  \left(\vartheta ,0,\tau\right)  + y_0 e^{\vartheta
   _0t} \int_{\tau }^{t}\exp\left(- \frac{f^2}{\sigma ^{2}}\int_{s}^{t}\gamma
  ^*\left(\vartheta_0 ,r\right){\rm d}r\right){\rm d}s\\
&= y_0e^{\vartheta _0\tau }\tau \Phi^*  \left(\vartheta ,0,\tau\right)  + y_0\;\bar H_t\left(\vartheta_0
 ,\vartheta _0\right)=y_0 \,\bar  G_t\left(\vartheta ,\vartheta _0\right). 
\end{align*}

 These limits are proved for the integrals on the intervals
$\left[\tau,T\right]$, but the value $\tau$ can be taken close to 0.  It can
be verified that the integral on the interval $\left[0,\tau\right]$ can be done
less than any constant by choosing sufficiently small $\tau$.  That is why in
the expressions for the limits we use the integrals on the intervals $(0,T]$.

The Fisher information we write in two forms: for $\varepsilon >0$ and $\varepsilon =0$
\begin{align*}
{\rm I}_{\tau,\varepsilon  }^t\left(\vartheta \right)&=\frac{f^2}{\sigma
  ^2}\int_{\tau }^{t} \dot m_*\left(\vartheta ,s\right)^2\,{\rm d}s,\qquad
\tau <t\leq T ,   \\
{\rm I}_{\tau }^t\left(\vartheta_0 \right)&=y_0^2\,J_\tau ^t\left(\vartheta _0\right),
\end{align*}
where $J_\tau ^t\left(\vartheta _0\right) $ is the corresponding deterministic
function. 
We have ${\rm I}_{\tau,\varepsilon  }^t\left(\vartheta_0 \right)
\longrightarrow  {\rm I}_{\tau }^t\left(\vartheta_0 \right) $ as $\varepsilon
\rightarrow 0$.

The One-step MLE-process $\vartheta _{t,\varepsilon }^\star,\tau <t\leq T $    can be defined by the equality
\begin{align*}
\vartheta _{t,\varepsilon }^\star=\vartheta _{\tau_* ,\varepsilon }^*+{\rm
  I}_{\tau_*,\varepsilon  }^t\left(\vartheta _{\tau ,\varepsilon }^* \right)^{-1}\frac{f}{\sigma ^2}\int_{\tau }^{t}
{ \dot m_*(\vartheta_{\tau_* ,\varepsilon }^*,s )}{ }\;\left[{\rm
      d}X_s-fm_*(\vartheta_{\tau_* ,\varepsilon }^*  ,s){\rm d}s\right]   .
\end{align*}
Here the MME $\vartheta _{\tau_* ,\varepsilon }^* $ is calculated for the values
$t_2=\frac{3\tau}{4} $ and $t_1=\frac{\tau}{2} $. The random processes $
m_*(\vartheta_{\tau_* ,\varepsilon }^*,s ),\tau \leq t\leq T $ and $\dot
m_*(\vartheta_{\tau_* ,\varepsilon }^*,s ),\tau \leq t\leq T $ are obtained as
solutions of \eqref{21-75b}  and \eqref{21-76} subject to the initial values
calculated as solutions of the same equations but written without stochastic
integrals as it was done in \eqref{21-z22}.

 We have the representation
\begin{align*}
\frac{\vartheta _{t,\varepsilon }^\star-\vartheta _0}{\varepsilon
}&=\frac{\vartheta _{\tau_* ,\varepsilon }^*-\vartheta _0}{\varepsilon}+{\rm 
  I}_{\tau_*,\varepsilon  }^t\left(\vartheta _{\tau ,\varepsilon }^* \right)^{-1}\frac{f}{\sigma }\int_{\tau }^{t}
{ \dot m_*(\vartheta_{\tau_* ,\varepsilon }^*,s )}{ }\;{\rm
      d}\bar W_s  \\
&\qquad + {\rm   I}_{\tau_*,\varepsilon  }^t\left(\vartheta _{\tau ,\varepsilon }^*
\right)^{-1}\frac{f^2}{\varepsilon \sigma^2 }\int_{\tau }^{t} 
{ \dot m_*(\vartheta_{\tau_* ,\varepsilon }^*,s )}{ }\;\left[m(\vartheta_0
  ,s)-m_*(\vartheta_{\tau_* ,\varepsilon }^*  ,s)\right] {\rm d}s. 
\end{align*}

Following the main steps of the  proof   of the Proposition \ref{P21-z1}  it
can be shown that this estimator for any $t\in (\tau ,T]$ is asymptotically
  conditionally normal:
\begin{align*}
\hat\eta _{t,\varepsilon }\left(\vartheta _0\right)&= \frac{\vartheta _{t,\varepsilon }^\star-\vartheta _0 }{\varepsilon
 }\Longrightarrow \hat \eta _t\left(\vartheta _0\right)={\rm 
  I}_{\tau }^t\left(\vartheta _{0} \right)^{-1}\frac{f}{\sigma }\int_{\tau }^{t}
{ \dot y(\vartheta_{0},\vartheta_{0},s )}{ }\;{\rm
      d}w\left(s\right)  ,\\
&= y_0^{-1}\,   {\rm 
  J}_{\tau }^t\left(\vartheta _{0} \right)^{-1}\frac{f}{\sigma }\int_{\tau }^{t}
{\bar G_s\left(\vartheta _0,\vartheta _0\right)}{ }\;{\rm
      d}w\left(s\right)\equiv   y_0^{-1}\, \zeta _\vartheta \left(\vartheta _0,t\right) ,
\end{align*}
where  $\zeta _m\left(\vartheta _0,t\right),\tau \leq t\leq T $ is a Gaussian
process. 

We have as well the following asymptotic normality
\begin{align*}
\frac{y_0\left(\vartheta _{t,\varepsilon }^\star-\vartheta _0\right) }{\varepsilon
 }\Longrightarrow \zeta _\vartheta \left(\vartheta _0,t\right) \sim{\cal N}\left(0, {\rm J}_{\tau }^t\left(\vartheta_0
 \right)^{-1}\right),\qquad \quad \tau <t\leq T. 
\end{align*}

\subsubsection{Adaptive filter}

The  adaptive Kalman filter  is 
\begin{align*}
{\rm d}m_{t,\varepsilon }^\star&=\left[\vartheta_{t,\varepsilon }^\star-{f^2\gamma^*
    \left(\vartheta_{t,\varepsilon }^\star ,t\right)}{\sigma ^{-2}} \right]  m_{t,\varepsilon }^\star{\rm 
  d}t+{f\gamma^* \left(\vartheta_{t,\varepsilon }^\star ,t\right)}{\sigma ^{-2}}{\rm
    d}X_t,\quad \tau  \leq t\leq T.
\end{align*}
The initial value $m_{\tau ,\varepsilon }^\star=m\left(\vartheta_{\tau
  ,\varepsilon }^\star,\tau \right)$ is calculated as it was explained above
in \eqref{21-z22}.
\begin{proposition}
\label{P21-z3} The error of approximation $m_{t,\varepsilon
}^\star-m\left(\vartheta _0,t\right) $ has the following\\ asymptotics
\begin{align*}
\varepsilon ^{-1}\left(m_{t,\varepsilon }^\star-m\left(\vartheta
_0,t\right)\right) \Longrightarrow \xi _m^*\left(\vartheta _0,t\right),\qquad \tau \leq t\leq T,
\end{align*}
where the Gaussian process $\xi _m^*\left(\vartheta _0,t\right) $ defined
in \eqref{21-z34} below does not depend on $y_0$.

\end{proposition}
\begin{proof}

This equation and the equation \eqref{21-75b} with $\vartheta =\vartheta _0$
allow us to write the equation for the error of estimation $\delta
_m\left(t\right)=m_{t,\varepsilon }^\star-m\left(\vartheta _0,t\right)$
\begin{align*}
\delta _m\left(t\right)&=\delta _m\left(\tau \right)  \Phi^* \left(\vartheta
_0,0,\tau\right)+\int_{\tau }^{t}\Phi \left(\vartheta _0,s,t\right)
\left[A\left(\vartheta _{t,\varepsilon }^\star,s\right)- A\left(\vartheta
  _0,s\right)\right]m_{s,\varepsilon }^\star{\rm d}s\\
&\qquad  +f\sigma ^{-2}\int_{\tau }^{t}\Phi^* \left(\vartheta _0,s,t\right)
\left[\gamma^*\left(\vartheta _{s,\varepsilon }^\star,s\right)-
  \gamma^*\left(\vartheta _0,s\right)\right]{\rm d}X_s .
\end{align*}

Recall that the difference between $\gamma \left(\vartheta ,t\right)$ and $
\varepsilon ^2\gamma ^*\left(\vartheta ,t\right)$ is of order $\varepsilon
^2$. Therefore the difference  $m^*\left(\vartheta
_0,t\right)-m_*\left(\vartheta _0,t\right)$ is of the same order. 
We have
\begin{align*}
\delta_m\left(t\right) &=\dot m\left(\vartheta _0,\tau \right)\delta_\vartheta \left(\tau
\right) \Phi^* \left(\vartheta
_0,0,\tau \right) +\int_{\tau }^{t}\Phi^* \left(\vartheta _0,s,t\right)
\dot  A\left(\vartheta
  _0,s\right) m_{s,\varepsilon }^\star  \delta_\vartheta
\left(s\right)  {\rm d}s\\ 
&\qquad  +f\sigma ^{-2}\int_{\tau }^{t}\Phi^* \left(\vartheta _0,s,t\right)
\dot 
  \gamma^*\left(\vartheta _0,s\right)\delta_\vartheta
\left(s\right) {\rm d}X_s +O\left(\varepsilon ^2\right)
\end{align*}
where $\delta_\vartheta
\left(s\right)=\vartheta _{s,\varepsilon }^\star-\vartheta _0 $. Therefore
\begin{align}
\label{21-z34}
\varepsilon ^{-1}&\delta_m\left(t\right) =\dot m\left(\vartheta _0,\tau
\right) \Phi^* \left(\vartheta
_0,0,\tau\right)\varepsilon ^{-1}\delta_\vartheta \left(\tau 
\right) +\int_{\tau }^{t}\Phi^* \left(\vartheta _0,s,t\right)
\dot  A\left(\vartheta
  _0,s\right) m_{s,\varepsilon }^\star \varepsilon ^{-1} \delta_\vartheta
\left(s\right)  {\rm d}s\nonumber\\ 
&\qquad\qquad  +f\sigma ^{-2}\int_{\tau }^{t}\Phi^* \left(\vartheta _0,s,t\right)
\dot 
  \gamma^*\left(\vartheta _0,s\right)\varepsilon ^{-1}\delta_\vartheta
\left(s\right) {\rm d}X_s +O\left(\varepsilon \right)\nonumber\\
&\Longrightarrow \dot y\left(\vartheta _0,\vartheta _0,\tau \right) \Phi^*
\left(\vartheta _0,0,\tau\right)\eta _{\tau
}\left(\vartheta _0\right) +\int_{\tau }^{t}\Phi^* \left(\vartheta _0,s,t\right) 
\dot  A\left(\vartheta
  _0,s\right) y_s\left(\vartheta _0\right) \eta _s\left(\vartheta _0\right)
      {\rm d}s\nonumber\\  
&\qquad\qquad  +f^2\sigma ^{-2}\int_{\tau }^{t}\Phi^* \left(\vartheta _0,s,t\right)
\dot  
  \gamma^*\left(\vartheta _0,s\right)y_s\left(\vartheta _0\right) \eta
  _s\left(\vartheta _0\right) {\rm d}s  \nonumber\\  
&=G_\tau \left(\vartheta _0,\vartheta _0\right)\Phi^* \left(\vartheta
_0,0,\tau\right)\zeta  \left(\vartheta _0,\tau \right)\nonumber\\
&\qquad +\int_{\tau }^{t}\Phi^*
  \left(\vartheta _0,s,t\right)  
\dot  A\left(\vartheta
  _0,s\right) H_s\left(\vartheta _0,\vartheta _0\right) \zeta  \left(\vartheta
_0,s \right)      {\rm d}s\nonumber\\
&\qquad \qquad +f^2\sigma ^{-2}\int_{\tau }^{t}\Phi^* \left(\vartheta _0,s,t\right)
\dot  
  \gamma^*\left(\vartheta _0,s\right)H_s\left(\vartheta _0,\vartheta _0\right)
  \zeta  \left(\vartheta _0,s \right)      {\rm d}s\nonumber\\  
&=G_\tau \left(\vartheta _0,\vartheta _0\right)\Phi^* \left(\vartheta
_0,0,\tau\right)\zeta  \left(\vartheta _0,\tau \right)+\int_{\tau }^{t}\Phi^*
  \left(\vartheta _0,s,t\right)  
 H_s\left(\vartheta _0,\vartheta _0\right) \zeta  \left(\vartheta
_0,s \right)      {\rm d}s\nonumber\\  
&\equiv \xi  _m^*\left(\vartheta _0,t\right) .
\end{align}
\end{proof}

\section{Discussion}
We considered three relatively simple models. The obtained results can be
generalized for more general, say, multidimensional partially observed
systems. The principal question in each model is: how to construct the
preliminary estimator.

Remark, that in the third problem (random initial value) the family of
measures of this statistical experiment with unknown parameter $\vartheta $ is
locally asymptotically mixing normal
\index{family of measures!locally asymptotically mixing normal} (LAMN) (see
Remark \ref{R21-15})
The study of MLE  can be an interesting problem but we will not do it here.

 Of course
there is a problem of definition of asymptotic efficiency of adaptive filter in the problem of
estimation of  $m\left(\vartheta _0,t\right)$. Remind that in the construction of
the lower bound we obtain on the right hand side the expression like 
$$
\Ex_\vartheta\left[ \dot m\left(\vartheta ,t\right) {\rm I}_0^t\left(\vartheta \right)^{-1}
\Delta _\varepsilon \left(\vartheta ,X^t\right)\right]^2
$$ 
where  the random variable   $y_0$  vanishes in  the limit. This probably can be
proved but we do not occupy by this problem here. In any case the proposed
here One-step MLE-process with fixed $\tau $ is not asymptotically efficient.

Hence the MLE $\hat\vartheta _\varepsilon $ can be asymptotically conditionally
normal, but in the problem of estimation  $m\left(\vartheta _0,t\right)$ the
asymptotically efficient estimator can be asymptotically normal. Of course
there is a problem of definition of asymptotic efficiency in the problem of
estimation $m\left(\vartheta _0,t\right)$ (see \cite{Kut24a}-\cite{Kut24c}
where the question of asymptotic optimality of the adaptive filters was discussed).

{\bf Acknowledgments.}  I am gratefull to  Reviwers for usefull comments.
This research was financially supported by the Russian Science Foundation
research project No 24-11-00191.



\end{document}